\documentclass[3p,10pt]{elsarticle}

\usepackage{amssymb,amsthm,siunitx,tikz}
\usepackage{latexsym,amsmath,subfigure,color,pgfplots}
\usepackage{lineno}

\journal{}

\newcommand{\eps}{\varepsilon}
\newcommand{\kb}{k_{\mathrm{b}}}
\newcommand{\epsb}{\varepsilon_{\mathrm{b}}}
\newcommand{\mub}{\mu_{\mathrm{b}}}
\newcommand{\sigmab}{\sigma_{\mathrm{b}}}
\newcommand{\set}[1]{\left\{#1\right\}}

\newcommand{\p}{\partial}
\newcommand{\rd}{\mathrm{d}}
\newcommand{\ma}{\mathbf{a}}
\newcommand{\mb}{\mathbf{b}}
\newcommand{\mr}{\mathbf{r}}

\newcommand{\mA}{\mathcal{A}}
\newcommand{\mB}{\mathcal{B}}
\newcommand{\mE}{\mathbf{E}}
\newcommand{\mF}{\mathbf{F}}
\newcommand{\mG}{\mathbf{G}}
\newcommand{\mH}{\mathbf{H}}
\newcommand{\vt}{\boldsymbol{\theta}}
\newcommand{\vv}{\boldsymbol{\vartheta}}

\DeclareMathOperator*{\inc}{inc}
\DeclareMathOperator*{\osm}{OSM}
\DeclareMathOperator*{\osmm}{OSMM}
\DeclareMathOperator*{\mosm}{MOSM}
\DeclareMathOperator*{\scat}{scat}

\DeclareMathOperator*{\area}{area}

\theoremstyle{plain}
\newtheorem{theorem}{Theorem}[section]
\newtheorem{lemma}{Lemma}[section]

\theoremstyle{remark}
\newtheorem{remark}{Remark}[section]

\begin{document}

\begin{frontmatter}



\title{Inversion of limited-aperture Fresnel experimental data using orthogonality sampling method with single and multiple sources}
%On the imaging of short sound-soft open arc via orthogonality sampling method

\author{Won-Kwang Park}
\ead{parkwk@kookmin.ac.kr}
\address{Department of Information Security, Cryptography, and Mathematics, Kookmin University, Seoul, 02707, Korea}

\begin{abstract}
In this study, we consider the application of orthogonality sampling method (OSM) with single and multiple sources for a fast identification of small objects in limited-aperture inverse scattering problem. We first apply the OSM with single source and show that the indicator function with single source can be expressed by the Bessel function of order zero of the first kind, infinite series of Bessel function of nonzero integer order of the first kind, range of signal receiver, and the location of emitter. Based on this result, we explain that the objects can be identified through the OSM with single source but the identification is significantly influenced by the location of source and applied frequency. For a successful improvement, we then consider the OSM with multiple sources. Based on the identified structure of the OSM with single source, we design an indicator function of the OSM with multiple sources and show that it can be expressed by the square of the Bessel function of order zero of the first kind an infinite series of the square of Bessel function of nonzero integer order of the first kind. Based on the theoretical results, we explain that the objects can be identified uniquely through the designed OSM. Several numerical experiments with experimental data provided by the Institute Fresnel demonstrate the pros and cons of the OSM with single source and how the designed OSM with multiple sources behave.
\end{abstract}


\end{frontmatter}





%\linenumbers

\section{Introduction}\label{sec:1}
Development of an effective and stable technique for retrieving unknown object from measured scattered field or scattering parameter data is an old but still interesting research subject to nowadays scientists and engineers because this subject is highly related to modern human life such as biomedical imaging \cite{ABM,ML2} including breast cancer detection \cite{CBC,SKLJS}, through-wall imaging for defect recognition \cite{B6,SSA}, damage detection of concrete structure \cite{FFK,WZZ}, land mine detection \cite{DEKPS,GCGGC}, synthetic-aperture radar imaging \cite{C2,GS}, ground penetrating radar \cite{LFNA,LSL}. We further refer to related studies \cite{A1,A2,BCS,C7,CK,N3,P-Book,Z2} for various applications. Let us notice that most of algorithms are based on Newton-type iteration schemes so that one must generate good initial guess which is close enough to the unknown objects.

Instead of iterative based algorithm, alternative non-iterative techniques have been investigated to retrieve unknown object. Throughout several researches about the bifocusing method \cite{JBRBTFC,KP4,SP2}, direct sampling method \cite{AHP2,IJZ1,KLAHP}, factorization method \cite{GYJH,LLP,P-FAC1}, MUltiple SIgnal Classificiation \cite{AILP,P-MUSIC6,ZC}, migration techniques \cite{AGKPS,P-SUB11,P-SUB18}, and topological derivative \cite{FPRV,LR1,P-TD5}, it has been turned out that although complete information of objects such as material properties cannot be retrieved, they are very effective techniques for identifying the existence, location, and outline shape of objects.

Orthogonality sampling method (OSM) is classified as a non-iterative imaging technique in both inverse scattering problem and microwave imaging. From the beginning study of the OSM by Potthast \cite{P1}, it has been applied various inverse scattering problems. Owing to the several studies \cite{ACP,BIPAC,G1,HN,KCP1,LNST}, it has been confirmed that the OSM is very fast, stable, and effective imaging technique in inverse scattering problem. Unfortunately, most of studies performed the numerical simulation to show the effectiveness of the OSM with synthetic data. In some researches \cite{BIPAC,LNST}, the OSM was applied in real-world inverse scattering problem with experimental datasets produced by the Institute Fresnel, France \cite{BS}. Although the OSM has demonstrated its applicability and robustness for retrieving a set of small objects from experimental data, an appropriate mathematical theory to explain some phenomena (for example, \textcircled{1} why the imaging performance is significantly on the location of source, \textcircled{2} why the application of low and high frequencies is not appropriate for retrieving multiple objects), and to design alternative technique for improving the imaging performance has not been established yet.

In this paper, we consider the application of the OSM for identifying a set of objects from experimental Fresnel data. First, we introduce the traditional indicator function for the OSM and reveal its mathematical structure by establishing a relationship with the Bessel function of order zero of the first kind, infinite series of Bessel function of nonzero integer order of the first kind, range of signal receiver, and the location of emitter. Based on the structure, we explain some intrinsic properties of the OSM and  provide theoretical answers to the unexplained phenomena mentioned above. We then exhibit simulation results with experimental data to demonstrate the theoretical result and fundamental limitation of object detection.

Next, we consider the OSM with multiple sources to improve the imaging performance for a proper detection of objects. To this end, we adopt the traditional indicator function with multiple sources introduced in \cite{P1} and propose another indicator function. In order to show the applicability, effectiveness, improvement of the proposed indicator function, and unique determination, we show that it can be expressed by the square of the Bessel function of order zero of the first kind an infinite series of the square of Bessel function of nonzero integer order of the first kind. We then exhibit simulation results to support established structure, discovered certain properties of the designed indicator function, and compare the imaging/detection performances.

The rest of this paper is organized as follows. In Section \ref{sec:2}, we briefly survey the direct scattering problem in the presence of a set of small objects and introduce the traditional indicator function of the OSM. In Section \ref{sec:3},  mathematical structure of the indicator function with single source is explored by establishing a relationship with an infinite series of the Bessel functions, range of receivers, and the location of emitter. In Section \ref{sec:4}, a set of simulation results with experimental Fresnel dataset is exhibited to confirm the theoretical result and examine the influence of the location of emitter and frequencies at operation. In Section \ref{sec:5}, we introduce the traditional and design a new indicator functions with multiple sources, establish mathematical structure, discover some intrinsic properties of the designed indicator function including unique determination., and exhibit simulation results. Conclusions and perspectives are included in Section \ref{sec:6}.

\section{Direct scattering problem and orthogonality sampling method}\label{sec:2}
Let $\Omega$ be a two-dimensional homogeneous region, $D_s\subset\Omega$, $s=1,2,\ldots,S$, be a two-dimensional small object, and $D$ be the collection of $D_s$. Throughout this paper, we assume that all $D_s$ are well-separated from each other and $\Omega$ is a subset of interior of an anechoic chamber so that the values of background conductivity, permeability, and permittivity are set to $\sigmab\approx0$, $\mub=4\pi\times \SI{e-7}{\henry/\meter}$, and $\eps_0=\SI{8.854e-12}{\farad/\meter}$, respectively, refer to \cite{P-Book}. Correspondingly, every $D_s$ and $\Omega$ are characterized by the value of dielectric permittivity at given angular frequency $\omega=2\pi f$. Let $\eps_s$ and $\epsb$ as the value of permittivity of $D_s$ and $\Omega$, respectively, and $\kb=\omega\sqrt{\epsb\mub}$ be the background wavenumber. With this, we introduce the following piecewise constant
\[\eps(\mr)=\left\{\begin{array}{ccl}
\eps_s&\text{for}&\mr\in D_s\\
\epsb&\text{for}&\mr\in\Omega\backslash\overline{D}.
\end{array}\right.\]

Let us denote $\ma_m$ and $\mb_n$ as the location of $m$th emitter $\mA_m$ and $n$th receiver $\mB_n$, respectively. Following to \cite{BS}, $\ma_m$ and $\mb_n$ can be written as
\[\ma_m=|\ma_m|(\cos\vartheta_m,\sin\vartheta_m)=|\ma_m|\vv_m\quad\text{with}\quad|\ma_m|\equiv|\ma|=\SI{0.72}{\meter},\quad\vartheta_m=\frac{2(n-1)\pi}{N}\]
and
\[\mb_n=|\mb_n|(\cos\theta_n,\sin\theta_n)=|\mb_n|\vt_n\quad\text{with}\quad|\mb_n|\equiv|\mb|=\SI{0.76}{\meter},\quad\theta_n=\vartheta_m+\frac{\pi}{3}+\frac{4(n-1)\pi}{3(N-1)},\]
respectively. Here, $\vv_m\in\mathbb{S}^1$ and $\vt_n\in\mathbb{S}_m^1$, where $\mathbb{S}^1$ denotes the unit circle centered at the origin, and
\[\mathbb{S}_m^1=\set{(\cos\theta,\sin\theta):\vartheta_m+\frac{\pi}{3}\leq\theta\leq\vartheta_m+\frac{5\pi}{3}}\subset\mathbb{S}^1.\]
For an illustration, we refer to Figure \ref{Configuration_Fresnel}. Then, the incident field at the fixed point source $\ma_m$ can be written as follows: for $\mr\in\Omega$,
\[u_{\inc}(\mr,\ma_m)=-\frac{i}{4}H_0^{(1)}(\kb|\mr-\ma_m|):=G(\mr,\ma_m),\]
where $H_0^{(1)}$ denotes the Hankel function of order zero of the first kind. Correspondingly, the time-harmonic total field $u(\mb_n,\mr)$ measured at $n$th receiver $\mb_n$ satisfies
\[\triangle u(\mb_n,\mr)+\omega^2\mub\eps(\mr)u(\mb_n,\mr)=0\quad\text{for}\quad\mr\in\Omega\]
with transmission condition $u(\mb_n,\mr)\big|_--u(\mb_n,\mr)\big|_+=0$ on $\p D_s$, $s=1,2,\ldots,S$. Here, the time harmonic $e^{-i\omega t}$ is assumed. Let $u_{\scat}(\mb_n,\mr)$ as the scattered-field corresponding to the incident field. Then based on \cite{CK}, $u_{\scat}(\mb_n,\mr)$ can be expressed by the single-layer potential with unknown density function $\varphi$:
\[u_{\scat}(\mb_n,\mr)=u(\mb_n,\ma_m)-u_{\inc}(\mr,\ma_m)=\int_D G(\mb_n,\mr)\varphi(\mr,\ma_m)\rd\mr.\]
Note that the closed form of the density function $\varphi(\mr,\ma_m)$ is unknown, it is not appropriate to use the $u_{\scat}(\mb_n,\mr)$ directly to design the indicator function of the OSM. Due to this reason, we use the following asymptotic expansion formula, which is the key formula to design and analyze the structure of the indicator function.

\begin{lemma}[Asymptotic expansion formula \cite{AK2,CV}] For sufficiently large $\omega$, $u_{\scat}(\mb_n,\mr)$ can be represented as
\begin{align}
\begin{aligned}\label{AsymptoticFormula}
u_{\scat}(\mb_n,\mr)&=\kb^2\int_D\left(\frac{\eps(\mr)-\epsb}{\epsb\mub}\right)G(\mb_n,\mr)u_{\inc}(\mr,\ma_m)\rd\mr\\
&=\kb^2\int_D\left(\frac{\eps(\mr)-\epsb}{\epsb\mub}\right)G(\mb_n,\mr)G(\mr,\ma_m)\rd\mr.
\end{aligned}
\end{align}
\end{lemma}

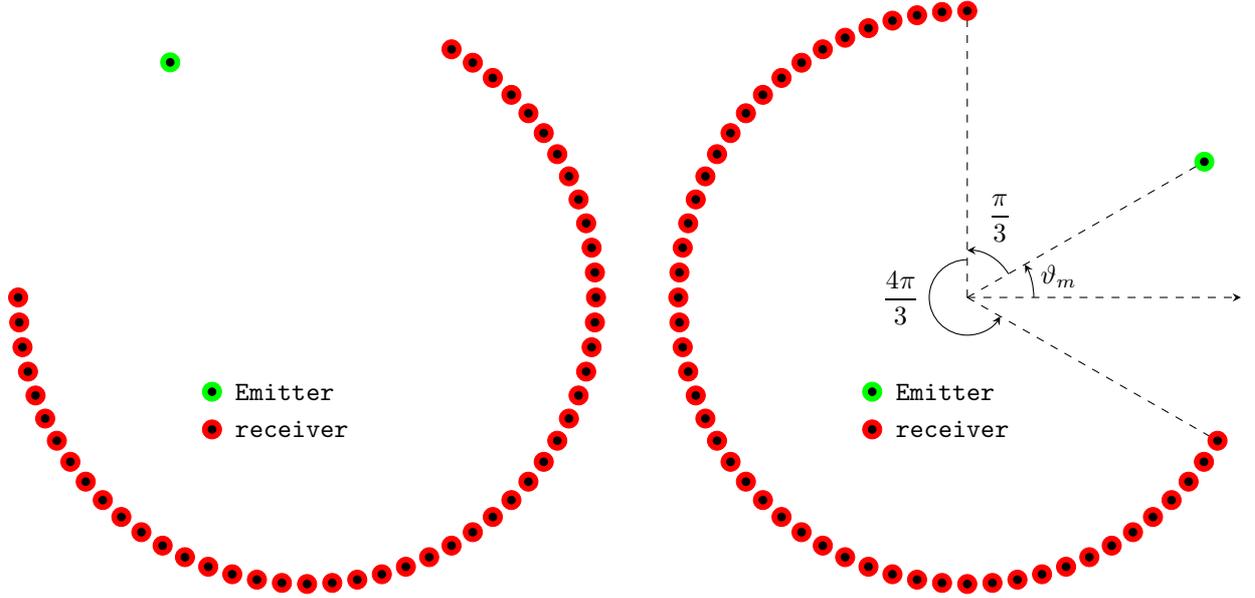
\begin{figure}[h]
\begin{center}
\subfigure{
\begin{tikzpicture}[scale=2.5]

% Antenna drawing
\def\RT{0.72*2};
\def\RR{0.76*2};
\def\Edge{0.15*2};

%\draw[black,dashed,-] (0,0) -- ({\RT*cos(0)},{\RT*sin(0)});
%\draw[black,dashed,-] (0,0) -- ({\RR*cos(60)},{\RR*sin(60)});
%\draw[black,dashed,-] (0,0) -- ({\RR*cos(300)},{\RR*sin(300)});

%\draw[black,solid,-stealth] ({0.25*cos(0)},{0.25*sin(0)}) arc (0:60:0.25);
%\node[black] at ({0.4*cos(30)},{0.4*sin(30)}) {$\displaystyle\frac{\pi}{6}$};

%\draw[black,solid,-stealth] ({0.2*cos(60)},{0.2*sin(60)}) arc (60:300:0.2);
%\node[black] at ({0.35*cos(180)},{0.4*sin(180)}) {$%\displaystyle\frac{5\pi}{3}$};

\draw[green,fill=green] ({\RT*cos(120)},{\RT*sin(120)}) circle (0.05cm);
\draw[black,fill=black] ({\RT*cos(120)},{\RT*sin(120)}) circle (0.02cm);

%\foreach \alpha in {10,20,...,350}
%{\draw[green!10!white,fill=green!10!white] ({\RT*cos(\alpha)},{\RT*sin(\alpha)}) circle (0.05cm);
%\draw[black!10!white,fill=black!10!white] ({\RT*cos(\alpha)},{\RT*sin(\alpha)}) circle (0.02cm);}
\foreach \beta in {180,185,...,420}
{\draw[red,fill=red] ({\RR*cos(\beta)},{\RR*sin(\beta)}) circle (0.05cm);
\draw[black,fill=black] ({\RR*cos(\beta)},{\RR*sin(\beta)}) circle (0.02cm);}

% ROI drawing
%\draw (-\Edge,-\Edge) -- (-\Edge,\Edge) -- (\Edge,\Edge) -- (\Edge,-\Edge) -- cycle;
%\node at (\Edge-0.1,-\Edge+0.1) {$\Omega$};

% Antennas Legend
\draw[green,fill=green] (-0.5,-0.5) circle (0.05cm);
\draw[black,fill=black] (-0.5,-0.5) circle (0.02cm);
\node[right] at (-0.5,-0.5) {\texttt{~Emitter}};
\draw[red,fill=red] (-0.5,-0.7) circle (0.05cm);
\draw[black,fill=black] (-0.5,-0.7) circle (0.02cm);
\node[right] at (-0.5,-0.7) {\texttt{~receiver}};

\end{tikzpicture}}\qquad
\subfigure{
\begin{tikzpicture}[scale=2.5]

% Antenna drawing
\def\RT{0.72*2};
\def\RR{0.76*2};
\def\Edge{0.15*2};

\draw[black,dashed,-stealth] (0,0) -- ({\RT*cos(0)},{\RT*sin(0)});
\draw[black,dashed,-] (0,0) -- ({\RT*cos(30)},{\RT*sin(30)});
\draw[black,dashed,-] (0,0) -- ({\RR*cos(90)},{\RR*sin(90)});
\draw[black,dashed,-] (0,0) -- ({\RR*cos(330)},{\RR*sin(330)});

\draw[black,solid,-stealth] ({0.25*cos(30)},{0.25*sin(30)}) arc (30:90:0.25);
\node[black] at ({0.5*cos(70)},{0.45*sin(70)}) {$\displaystyle\frac{\pi}{3}$};

\draw[black,solid,-stealth] ({0.2*cos(90)},{0.2*sin(90)}) arc (90:330:0.2);
\node[black] at ({0.35*cos(180)},{0.4*sin(180)}) {$\displaystyle\frac{4\pi}{3}$};

\draw[black,solid,-stealth] ({0.35*cos(0)},{0.35*sin(0)}) arc (0:30:0.35);
\node[black] at ({0.5*cos(30)},{0.22*sin(30)}) {$~~\vartheta_m$};

\draw[green,fill=green] ({\RT*cos(30)},{\RT*sin(30)}) circle (0.05cm);
\draw[black,fill=black] ({\RT*cos(30)},{\RT*sin(30)}) circle (0.02cm);

\foreach \beta in {90,95,...,330}
{\draw[red,fill=red] ({\RR*cos(\beta)},{\RR*sin(\beta)}) circle (0.05cm);
\draw[black,fill=black] ({\RR*cos(\beta)},{\RR*sin(\beta)}) circle (0.02cm);}

% ROI drawing
%\draw (-\Edge,-\Edge) -- (-\Edge,\Edge) -- (\Edge,\Edge) -- (\Edge,-\Edge) -- cycle;
%\node at (\Edge-0.1,-\Edge+0.1) {$\Omega$};

% Antennas Legend
\draw[green,fill=green] (-0.5,-0.5) circle (0.05cm);
\draw[black,fill=black] (-0.5,-0.5) circle (0.02cm);
\node[right] at (-0.5,-0.5) {\texttt{~Emitter}};
\draw[red,fill=red] (-0.5,-0.7) circle (0.05cm);
\draw[black,fill=black] (-0.5,-0.7) circle (0.02cm);
\node[right] at (-0.5,-0.7) {\texttt{~receiver}};

\end{tikzpicture}}
\caption{\label{Configuration_Fresnel}Illustration of  measurement configuration corresponding to the location of emitter.}
\end{center}
\end{figure}

\section{Indicator function with single source}\label{sec:3}
In this section, we consider the design an indicator function with single source. Let us denote $\mE(\ma_m)$ as the following arrangement of measurement data
\begin{equation}\label{ArrangementE}
\mE(\ma_m)=\Big(u_{\scat}(\mb_1,\mr),u_{\scat}(\mb_2,\mr),\ldots,u_{\scat}(\mb_N,\mr)\Big).
\end{equation}
Now, applying the mean-value theorem to \eqref{AsymptoticFormula} yields
\[u_{\scat}(\mb_n,\mr_s)=\sum_{s=1}^{S}\kb^2\area(D_s)\left(\frac{\eps_s-\epsb}{\epsb\mub}\right)G(\mb_n,\mr_s)G(\mr_s,\ma_m),\]
thus, we can design the indicator function of the OSM based on the orthogonality relation between the $u_{\scat}(\mb_n,\mr)$ and $G(\mb_n,\cdot)$. To this end, let us introduce a test vector: for each $\mr'\in\Omega$,
\[\mG(\mr')=\Big(G(\mb_1,\mr'),G(\mb_2,\mr'),\ldots,G(\mb_N,\mr')\Big)\]
and corresponding indicator function
\[\mathfrak{F}_{\osm}(\mr',\ma_m)=|\mE(\ma_m)\cdot\overline{\mG(\mr')}|=\left|\sum_{n=1}^{N}u_{\scat}(\mb_n,\mr)\overline{G(\mb_n,\mr')}\right|.\]
Then, map of $\mathfrak{F}_{\osm}(\mr',\ma_m)$ will contain peaks of large magnitude at $\mr'\in D_s$ thereby, it will be possible to recognize the existence or outline shape of $D_s$, $s=1,2,\ldots,S$. In order to discover the feasibility and some properties of the $\mathfrak{F}_{\osm}(\mr',\ma_m)$, we derive the following result.

\begin{theorem}\label{OSM_Single}
  Let $\vv_m=(\cos\vartheta_m,\sin\vartheta_m)$, $\vt_n=(\cos\theta_n,\sin\theta_n)$, $\vt=(\cos\theta,\sin\theta)$, and $\mr'-\mr=|\mr'-\mr|(\cos\phi,\sin\phi)$. Then, for sufficiently large $N$ and $\omega$, $\mathfrak{F}_{\osm}(\mr',\ma_m)$ can be represented as follows:
  \begin{equation}\label{Structure_Single}
    \mathfrak{F}_{\osm}(\mr',\ma_m)=\left|\frac{\kb}{6|\mb|}\int_D\left(\frac{\eps(\mr)-\epsb}{\epsb\mub}\right)G(\mr,\ma_m)\left[J_0(\kb|\mr'-\mr|)+\frac{3}{\pi}\mathcal{E}(\mr',\mr,\ma_m)\right]\rd\mr\right|,
  \end{equation}
where $J_p$ denotes the Bessel function of order $q$ and
\begin{equation}\label{DisturbFactor}
\mathcal{E}(\mr',\mr,\ma_m)=\sum_{p=1}^{\infty}\frac{(-i)^p}{p}J_p(\kb|\mr'-\mr|)\cos\big(p(\vartheta_m-\phi)\big)\sin\left(\frac{2p}{3}\pi\right).
\end{equation}
\end{theorem}
\begin{proof}
Since $4\kb|\mr-\mb_n|\gg1$ for $n=1,2,\cdots,N$, the following asymptotic form holds (see \cite{CK} for instance)
\begin{equation}\label{Asymptotic_Hankel}
H_0^{(1)}(\kb|\mb_n-\mr|)=\frac{(1-i)e^{i\kb|\mb_n|}}{\sqrt{\kb|\mb_n|\pi}}e^{-i\kb\vt_n\cdot\mr}.
\end{equation}
Thus, we can examine that
\[u_{\scat}(\mb_n,\mr)\approx-\frac{\kb^2(1+i)e^{i\kb|\mb_n|}}{4\sqrt{\kb|\mb_n|\pi}}\int_D\left(\frac{\eps(\mr)-\epsb}{\epsb\mub}\right)G(\mr,\ma_m)e^{-i\kb\vt_n\cdot\mr}\rd\mr\]
and correspondingly, we have
\begin{align*}
&\sum_{n=1}^{N}u_{\scat}(\mb_n,\mr)\overline{G(\mb_n,\mr')}\\
&\approx-\sum_{n=1}^{N}\left(\frac{\kb^2(1+i)e^{i\kb|\mb_n|}}{4\sqrt{\kb|\mb_n|\pi}}\int_D\left(\frac{\eps(\mr)-\epsb}{\epsb\mub}\right)G(\mr,\ma_m)e^{-i\kb\vt_n\cdot\mr}\rd\mr\right)\frac{(-1+i)e^{-i\kb|\mb_n|}}{4\sqrt{\kb|\mb_n|\pi}}e^{i\kb\vt_n\cdot\mr'}\\
&=\frac{\kb}{8|\mb|\pi}\int_D\left(\frac{\eps(\mr)-\epsb}{\epsb\mub}\right)G(\mr,\ma_m)\left(\sum_{n=1}^{N}e^{i\kb\vt_n\cdot(\mr'-\mr)}\right)\rd\mr.
\end{align*}
Since $N$ is sufficiently large, $\vt\cdot(\mr'-\mr)=|\mr'-\mr|\cos(\theta-\phi)$, and the following relation holds uniformly (see \cite{P-SUB3} for instance),
\begin{equation}\label{Jacobi-Anger}
\int_{\alpha}^{\beta}e^{ix\cos(\theta-\phi)}\rd\theta=(\beta-\alpha)J_0(x)+4\sum_{p=1}^{\infty}\frac{i^p}{p}J_p(x)\cos\frac{p(\beta+\alpha-2\phi)}{2}\sin\frac{p(\beta-\alpha)}{2},
\end{equation}
we can evaluate
\begin{align*}
\sum_{n=1}^{N}e^{i\kb\vt_n\cdot(\mr'-\mr)}&\approx\int_{\mathbb{S}_m^1}e^{i\kb\vt\cdot(\mr'-\mr)}\rd\vt\\
&=\int_{\theta_1=\vartheta_m+\pi/6}^{\theta_N=\vartheta_m+5\pi/6}e^{i\kb|\mr'-\mr|\cos(\theta-\phi)}\rd\theta\\%\limits_
&=\frac{4\pi}{3}J_0(\kb|\mr'-\mr|)+4\sum_{p=1}^{\infty}\frac{(-i)^p}{p}J_p(\kb|\mr'-\mr|)\cos\big(p(\vartheta_m-\phi)\big)\sin\left(\frac{2p}{3}\pi\right).
\end{align*}
Hence,
\[\sum_{n=1}^{N}u_{\scat}(\mb_n,\mr)\overline{G(\mb_n,\mr')}=\frac{\kb}{6|\mb|}\int_D\left(\frac{\eps(\mr)-\epsb}{\epsb\mub}\right)G(\mr,\ma_m)\left[J_0(\kb|\mr'-\mr|)+\frac{3}{\pi}\mathcal{E}(\mr',\mr,\ma_m)\right]\rd\mr\]
and correspondingly, \eqref{Structure_Single} can be derived.
\end{proof}

Based on the Theorem \ref{OSM_Single}, we can examine some properties of the indicator function.

\begin{remark}[Availability and limitation of object detection]\label{remark1}
Since $J_0(0)=1$ and $J_p(0)=0$ for $p=1,2,\ldots$, the resulting plot of indicator function $\mathfrak{F}_{\osm}(\mr',\ma_m)$ is expected to exhibit peaks of magnitudes $\frac{\kb}{6|\mb|}\left(\frac{\eps_s-\epsb}{\epsb\mub}\right)|G(\mr,\ma_m)|\area(D_s)$ at the $\mr'=\mr\in D_s$ sought.

Notice that the since $\mathfrak{F}_{\osm}(\mr',\ma_m)\propto|G(\mr,\ma_m)|$, imaging performance of the $\mathfrak{F}_{\osm}(\mr',\ma_m)$ will be significantly dependent on the position of the emitter $\mA_m$. Specially, if one applies extremely high frequency then the value of $|G(\mr,\ma_m)|$ becomes negligible because
\[|G(\mr,\ma_m)|=\left|\frac{(1-i)e^{i\kb|\mb_n|}}{\sqrt{\kb|\mb_n|\pi}}e^{-i\kb\vt_n\cdot\mr}\right|\longrightarrow0\quad\text{as}\quad\omega\longrightarrow\infty+.\]
Correspondingly, the value of $\mathfrak{F}_{\osm}(\mr',\ma_m)$ becomes negligible so that it will be unable to distinguish between unknown objects and several artifacts in the map of $\mathfrak{F}_{\osm}(\mr',\ma_m)$. Hence, we conclude that application of extremely high frequency does not guarantee the detection of unknown object through the OSM with single source.
\end{remark}

\begin{remark}[Detection of multiple objects]\label{remark2}
Suppose that there exists two objects $D_1$ and $D_2$ located at $\mr_1$ and $\mr_2$, respectively. Then, the following relation must satisfy to distinguish objects through the map of $\mathfrak{F}_{\osm}(\mr',\ma_m)$
\begin{equation}\label{Distinguish}
|\mr_1-\mr_2|>\frac{\lambda}{2},
\end{equation}
where $\lambda$ denotes positive wavelength.

In Section \ref{sec:4}, the scattered field data were generated in the presence of two circular shaped dielectric objects centered at $\mr_1=\SI{0.045}{\meter}$ and $\mr_2=-\SI{0.045}{\meter}$, i.e., $|\mr_1-\mr_2|=\SI{0.09}{\meter}$ (see \cite{BS} for instance). Notice that if $f=\SI{1}{\giga\hertz}$ then $\lambda/2=\SI{0.1499}{\meter}$ because $k=20.9585$ so that two objects cannot be distinguished through the map of $\mathfrak{F}_{\osm}(\mr',\ma_m)$ because $\lambda$ does not satisfies the relation \eqref{Distinguish}. If $f\geq\SI{2}{\giga\hertz}$ then $D_1$ and $D_2$ can be distinguished because $\lambda/2\leq\SI{0.0749}{\meter}$. Hence, we conclude that application of low frequency does not guarantee the detection of unknown objects through the OSM with single source.% Corresponding simulation results can be found in Section \ref{sec:4}.
\end{remark}

\begin{remark}[Effect of the factor $\mathcal{E}(\mr',\mr,\ma_m)$]\label{Remark3}
Based on the \eqref{DisturbFactor}, the factor $\mathcal{E}(\mr',\mr,\ma_m)$ not only does not contribute to the identification of objects but also disturb the identification by generating several artifacts. In order to examine the influence of the $\mathcal{E}(\mr',\mr,\ma_m)$, we consider the following quantity
\[\mathcal{D}_1(x)=\frac{3}{\pi}\left|\sum_{p=1}^{10^6}\frac{(-i)^p}{p}J_p(\kb|x|)\sin\left(\frac{2p}{3}\pi\right)\right|.\]
This is similar to the value of $3\mathcal{E}(\mr',\mr,\ma_m)/\pi$ for $\vartheta_m=0$ (i.e., $\ma_m=(1,0))$ and $\mr=(0,0)$. By comparing the oscillating properties with $J_0(\kb|x|)$, we can say that due to the oscillating properties of the $J_0(\kb|\mr'-\mr|)$ and $\mathcal{E}(\mr',\mr,\ma_m)$, maps of $\mathfrak{F}_{\osm}(\mr',\ma_m)$ will contain several artifacts and it will disturb the recognization of the existence of objects, refer to Figure \ref{FigureBessel1}.
\end{remark}

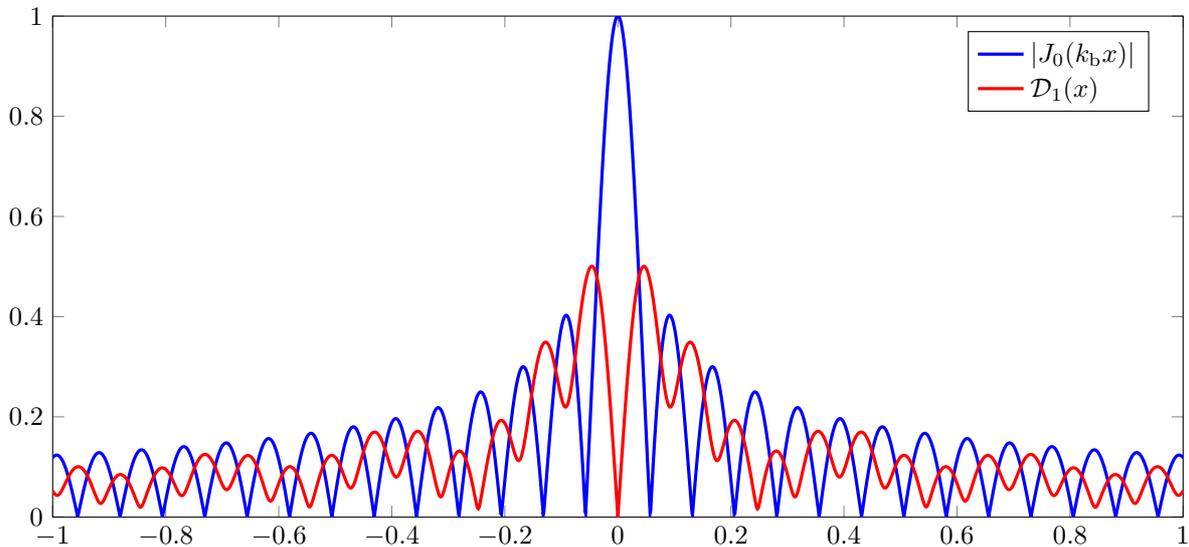
\begin{figure}[h]
\begin{center}
\begin{tikzpicture}
\begin{axis}
[legend style={fill=none},
width=\textwidth,
height=0.5\textwidth,
xmin=-1,
xmax=1,
ymin=0,
ymax=1,
legend pos=north east,
legend cell align={left}]
\addplot[line width=1.2pt,solid,color=blue] %
	table[x=x,y=y,col sep=comma]{BesselFunctions1.csv};
\addlegendentry{$|J_0(\kb x)|$};
\addplot[line width=1.2pt,solid,color=red] %
	table[x=x,y=z,col sep=comma]{BesselFunctions1.csv};
\addlegendentry{$\mathcal{D}_1(x)$};
\end{axis}
\end{tikzpicture}
\caption{\label{FigureBessel1} Plots of $|J_0(\kb x)|$ and $\mathcal{D}_1(x)$ for $-1\leq x\leq1$ and $f=\SI{2}{\giga\hertz}$.}
\end{center}
\end{figure}

\section{Simulation results with experimental data}\label{sec:4}
Here, we exhibit simulation results using experimental data \cite{BS}. The emitters and receivers are placed on the circles centered at the origin with radii $|\ma_m|=\SI{0.72}{\meter}$ and $|\mb_n|=\SI{0.76}{\meter}$, respectively, the the imaging region $\Omega$ was chosen as a square $(-\SI{0.1}{\meter},\SI{0.1}{\meter})\times(-\SI{0.1}{\meter},\SI{0.1}{\meter})$ to satisfy the relation $4\kb|\mr-\ma_m|$, $4\kb|\mr-\mb_n|\gg1$ for $m=1,2,\ldots,M$ and $n=1,2,\ldots,N$. The range of receivers is restricted from $\SI{60}{\degree}$ to $\SI{300}{\degree}$, with step size of $\SI{5}{\degree}$ based on each location of emitters. We refer to Figure \ref{Configuration_Fresnel} again for an illustration of measurement configuration.

The objects are composed of two filled dielectric cylinders $D_1$ and $D_2$ with circular cross section of radius $\SI{0.015}{\meter}$ and permittivity $\eps_s=(3\pm0.3)\epsb$ centered at $\mr_1=(0.045,0.010)\SI{}{\meter}$ and $\mr_2=(-0.045,0)\SI{}{\meter}$, respectively\footnote{In \cite{BS}, $\mr_1=(0.045,0)\SI{}{\meter}$ was given but throughout several results \cite{BT,BAB,CPR,KLAHP,KP4,MBQL}, accurate location of $D_1$ seems $\mr_1=(0.045,0.010)\SI{}{\meter}$.}. With this setting, we generated the imaging results $\mathfrak{F}_{\osm}(\mr',\ma_m)$ with $m=1$, $10$, and $25$, refer to Figure \ref{SimulationSetting} for illustration.

\begin{figure}[h]
\begin{center}
\subfigure[Simulation setup with $\mA_1$]{
\begin{tikzpicture}[scale=1.5]

% Antenna drawing
\def\RT{0.72*2};
\def\RR{0.76*2};
\def\Edge{0.15*2};
\def\BD{0.85*2};

\draw[green,fill=green] ({\RT*cos(0)},{\RT*sin(0)}) circle (0.05cm);
\draw[black,fill=black] ({\RT*cos(0)},{\RT*sin(0)}) circle (0.02cm);

\foreach \beta in {60,65,...,300}
{\draw[red,fill=red] ({\RR*cos(\beta)},{\RR*sin(\beta)}) circle (0.05cm);
\draw[black,fill=black] ({\RR*cos(\beta)},{\RR*sin(\beta)}) circle (0.02cm);}

% Antennas Legend
\draw[green,fill=green] (-0.7,0.15) circle (0.05cm);
\draw[black,fill=black] (-0.7,0.15) circle (0.02cm);
\node[right] at (-0.7,0.15) {\texttt{~emitter}};
\draw[red,fill=red] (-0.7,-0.15) circle (0.05cm);
\draw[black,fill=black] (-0.7,-0.15) circle (0.02cm);
\node[right] at (-0.7,-0.15) {\texttt{~receiver}};

% Boundary box
\draw[gray!50!white] (\BD,\BD) -- (\BD,-\BD) -- (-\BD,-\BD) -- (-\BD,\BD) -- cycle;
\end{tikzpicture}}\hfill
\subfigure[Simulation setup with $\mA_{10}$]{
\begin{tikzpicture}[scale=1.5]

% Antenna drawing
\def\RT{0.72*2};
\def\RR{0.76*2};
\def\Edge{0.15*2};
\def\BD{0.85*2};

\draw[green,fill=green] ({\RT*cos(90)},{\RT*sin(90)}) circle (0.05cm);
\draw[black,fill=black] ({\RT*cos(90)},{\RT*sin(90)}) circle (0.02cm);

\foreach \beta in {150,155,...,390}
{\draw[red,fill=red] ({\RR*cos(\beta)},{\RR*sin(\beta)}) circle (0.05cm);
\draw[black,fill=black] ({\RR*cos(\beta)},{\RR*sin(\beta)}) circle (0.02cm);}

% Antennas Legend
\draw[green,fill=green] (-0.7,0.15) circle (0.05cm);
\draw[black,fill=black] (-0.7,0.15) circle (0.02cm);
\node[right] at (-0.7,0.15) {\texttt{~emitter}};
\draw[red,fill=red] (-0.7,-0.15) circle (0.05cm);
\draw[black,fill=black] (-0.7,-0.15) circle (0.02cm);
\node[right] at (-0.7,-0.15) {\texttt{~receiver}};

% Boundary box
\draw[gray!50!white] (\BD,\BD) -- (\BD,-\BD) -- (-\BD,-\BD) -- (-\BD,\BD) -- cycle;
\end{tikzpicture}}
\hfill
\subfigure[Simulation setup with $\mA_{25}$]{
\begin{tikzpicture}[scale=1.5]

% Antenna drawing
\def\RT{0.72*2};
\def\RR{0.76*2};
\def\Edge{0.15*2};
\def\BD{0.85*2};

\draw[green,fill=green] ({\RT*cos(240)},{\RT*sin(240)}) circle (0.05cm);
\draw[black,fill=black] ({\RT*cos(240)},{\RT*sin(240)}) circle (0.02cm);

\foreach \beta in {300,305,...,540}
{\draw[red,fill=red] ({\RR*cos(\beta)},{\RR*sin(\beta)}) circle (0.05cm);
\draw[black,fill=black] ({\RR*cos(\beta)},{\RR*sin(\beta)}) circle (0.02cm);}

% Antennas Legend
\draw[green,fill=green] (-0.7,0.15) circle (0.05cm);
\draw[black,fill=black] (-0.7,0.15) circle (0.02cm);
\node[right] at (-0.7,0.15) {\texttt{~emitter}};
\draw[red,fill=red] (-0.7,-0.15) circle (0.05cm);
\draw[black,fill=black] (-0.7,-0.15) circle (0.02cm);
\node[right] at (-0.7,-0.15) {\texttt{~receiver}};

% Boundary box
\draw[gray!50!white] (\BD,\BD) -- (\BD,-\BD) -- (-\BD,-\BD) -- (-\BD,\BD) -- cycle;
\end{tikzpicture}}
\caption{\label{SimulationSetting}Illustration of antenna arrangement with transmitters $\mA_1$, $\mA_{10}$, and $\mA_{25}$.}
\end{center}
\end{figure}
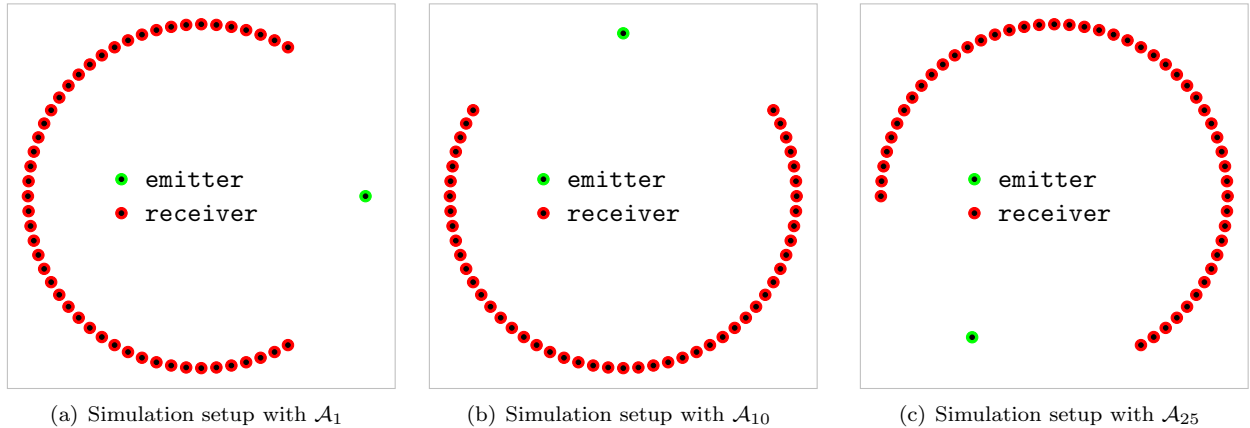

Figure \ref{Single1} shows maps of $\mathfrak{F}_{\osm}(\mr',\ma_1)$ with $\ma_1=\SI{0.72}{\meter}(\cos\SI{0}{\degree},\sin\SI{0}{\degree})$ at several frequencies. Based on the simulation results, it is impossible to recognize objects through the map of $\mathfrak{F}_{\osm}(\mr',\ma_1)$ when $f\leq\SI{2}{\giga\hertz}$ and $f\geq\SI{6}{\giga\hertz}$. Fortunately, peaks of large magnitudes are appeared at $\mr_1$ and $\mr_2$ so that we can recognize the existence of two objects but their outline shape cannot be determined.

\begin{figure}[h]
\begin{center}
\subfigure[$f=\SI{1}{\giga\hertz}$]{\includegraphics[width=.33\textwidth]{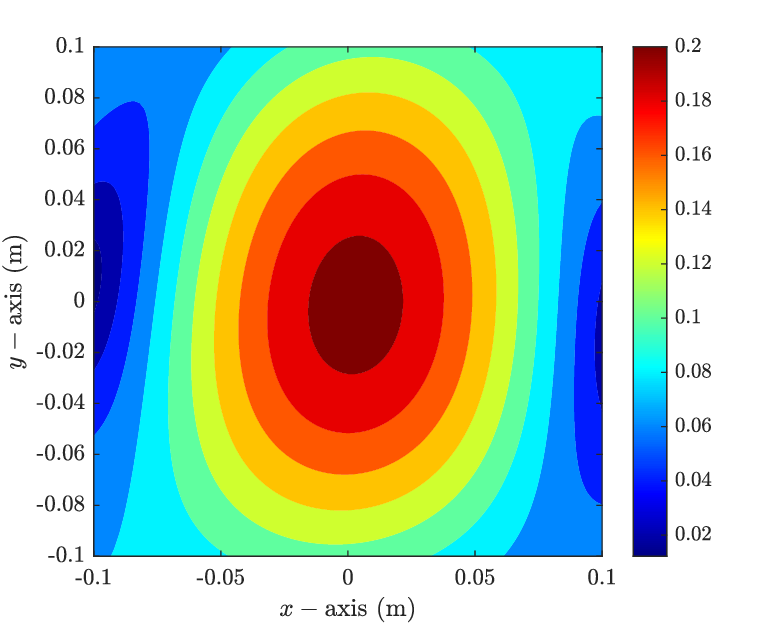}}\hfill
\subfigure[$f=\SI{2}{\giga\hertz}$]{\includegraphics[width=.33\textwidth]{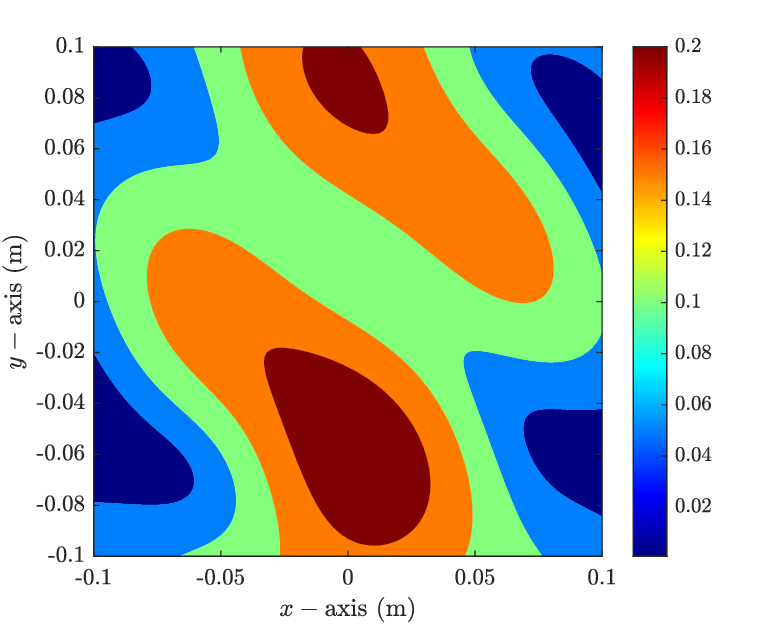}}\hfill
\subfigure[$f=\SI{3}{\giga\hertz}$]{\includegraphics[width=.33\textwidth]{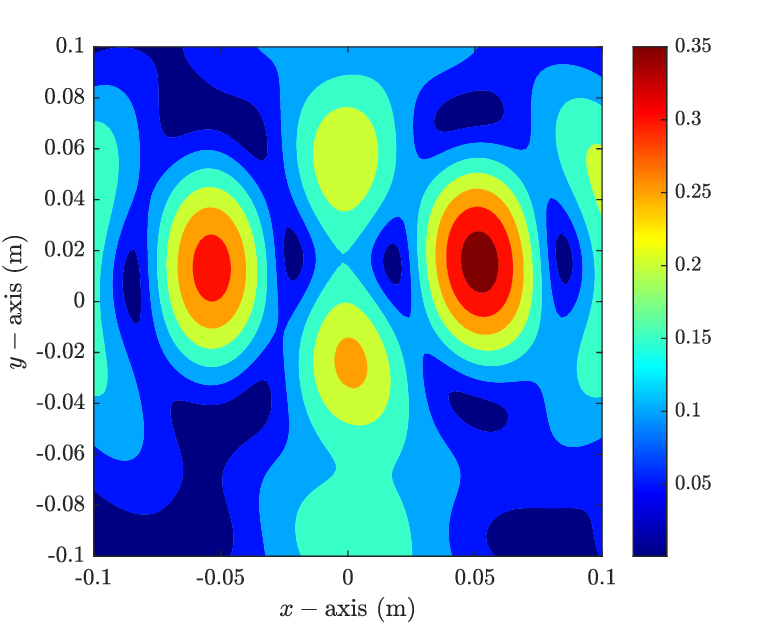}}\\
\subfigure[$f=\SI{4}{\giga\hertz}$]{\includegraphics[width=.33\textwidth]{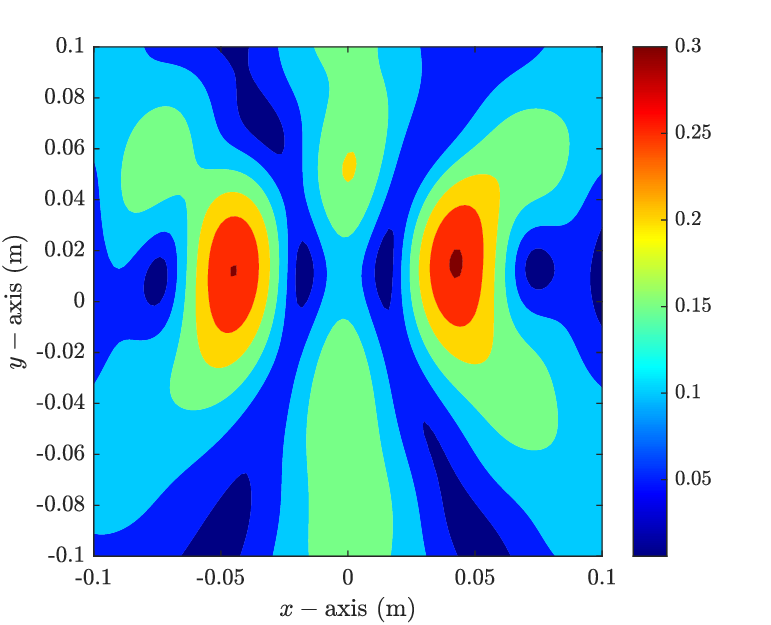}}\hfill
\subfigure[$f=\SI{6}{\giga\hertz}$]{\includegraphics[width=.33\textwidth]{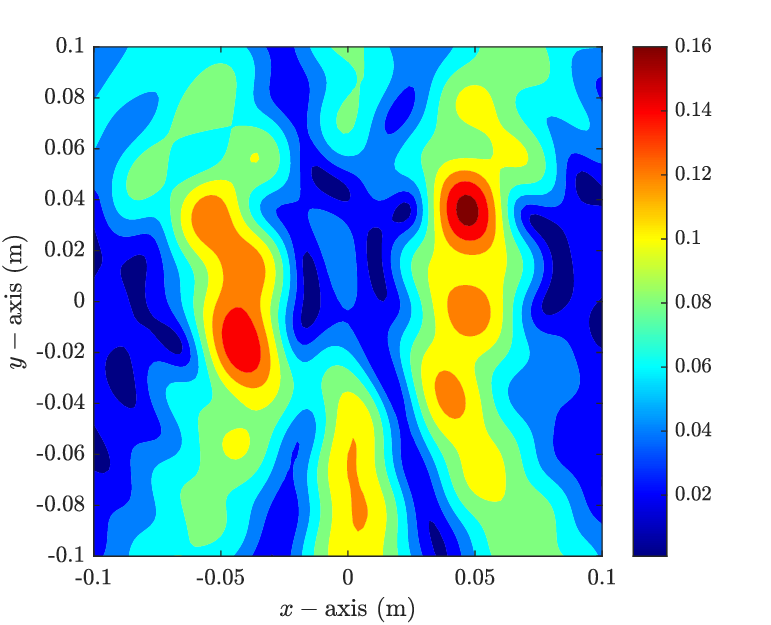}}\hfill
\subfigure[$f=\SI{8}{\giga\hertz}$]{\includegraphics[width=.33\textwidth]{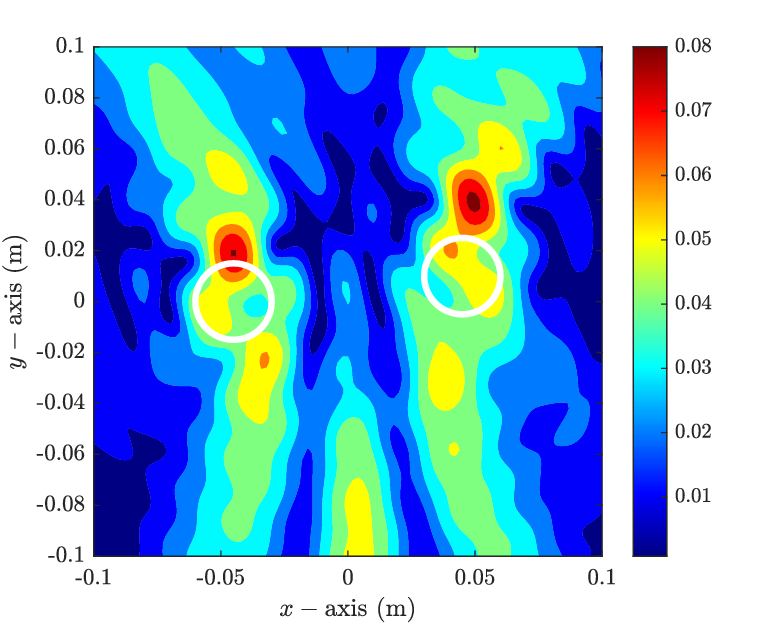}}
\caption{\label{Single1}Maps of $\mathfrak{F}_{\osm}(\mr',\ma_1)$. White colored circles describes the boundary of objects.}
\end{center}
\end{figure}

Figure \ref{Single10} shows maps of $\mathfrak{F}_{\osm}(\mr',\ma_{10})$ with $\ma_{10}=\SI{0.72}{\meter}(\cos\SI{90}{\degree},\sin\SI{90}{\degree})$ at several frequencies. In contrast to the results in Figure \ref{Single1}, it is possible to recognize the existence of two object at $f=\SI{2}{\giga\hertz}$ but it is still impossible to recognize them when $f=\SI{1}{\giga\hertz}$ and $f\geq\SI{6}{\giga\hertz}$. Moreover, it seems to be difficult to recognize the existence of objects when $f=\SI{4}{\giga\hertz}$ due to the appearance of two artifacts with large magnitudes.

\begin{figure}[h]
\begin{center}
\subfigure[$f=\SI{1}{\giga\hertz}$]{\includegraphics[width=.33\textwidth]{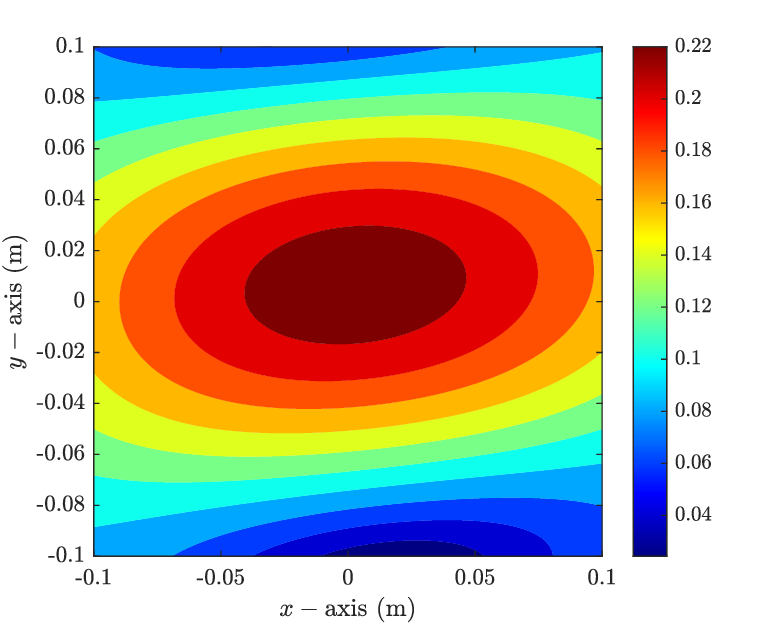}}\hfill
\subfigure[$f=\SI{2}{\giga\hertz}$]{\includegraphics[width=.33\textwidth]{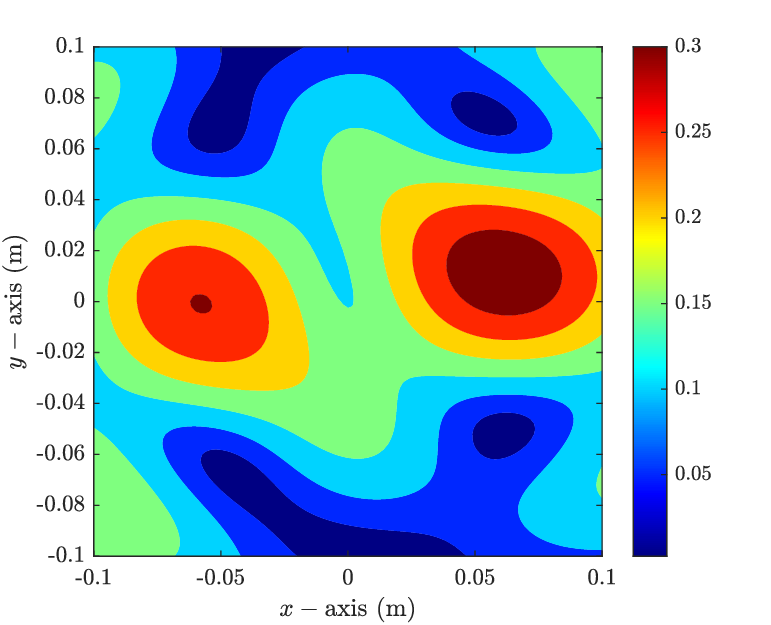}}\hfill
\subfigure[$f=\SI{3}{\giga\hertz}$]{\includegraphics[width=.33\textwidth]{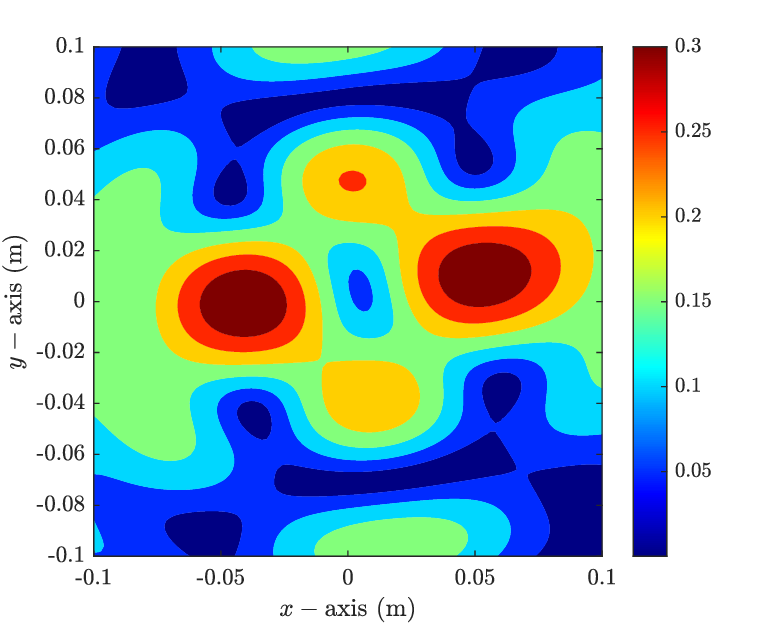}}\\
\subfigure[$f=\SI{4}{\giga\hertz}$]{\includegraphics[width=.33\textwidth]{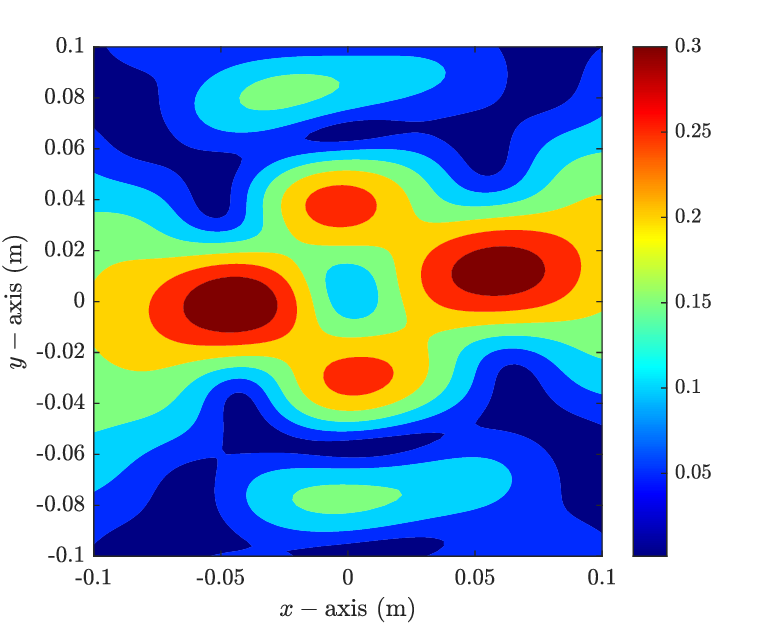}}\hfill
\subfigure[$f=\SI{6}{\giga\hertz}$]{\includegraphics[width=.33\textwidth]{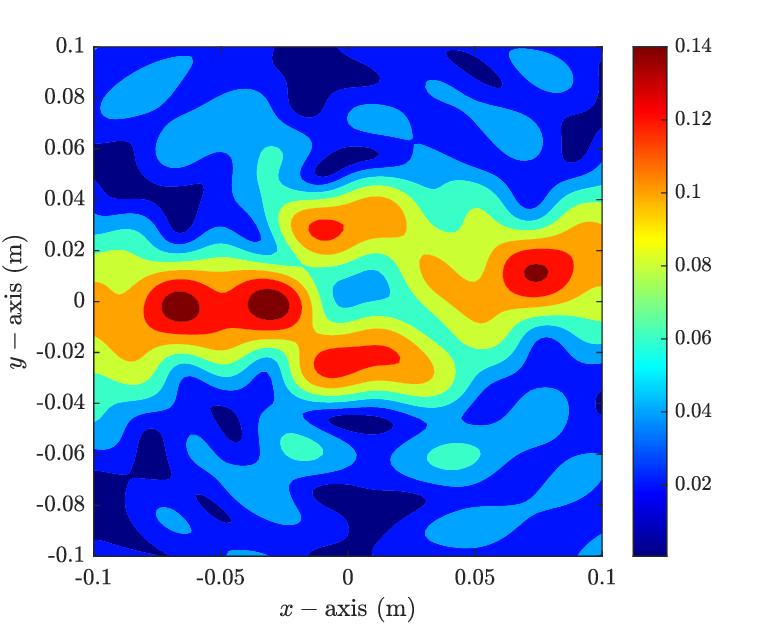}}\hfill
\subfigure[$f=\SI{8}{\giga\hertz}$]{\includegraphics[width=.33\textwidth]{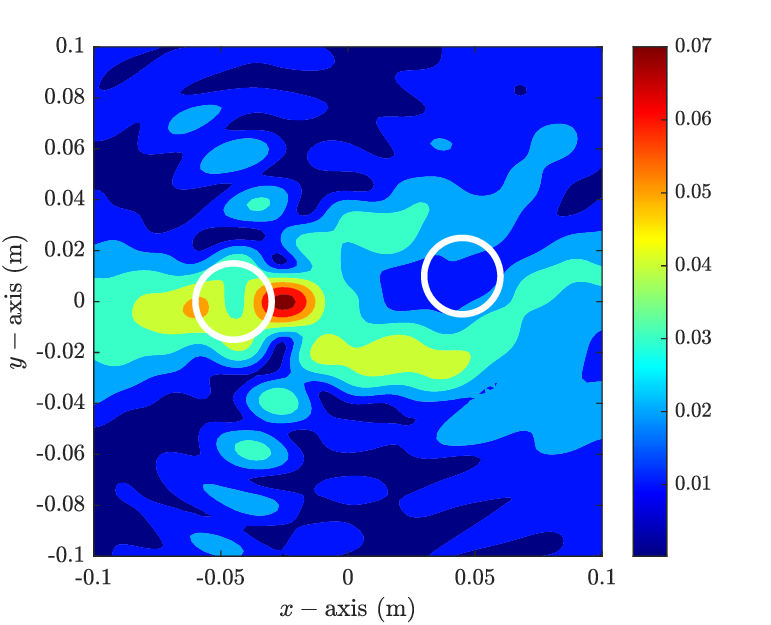}}
\caption{\label{Single10}Maps of $\mathfrak{F}_{\osm}(\mr',\ma_{10})$. White colored circles describes the boundary of objects.}
\end{center}
\end{figure}

Figure \ref{Single25} shows maps of $\mathfrak{F}_{\osm}(\mr',\ma_{25})$ with $\ma_{25}=\SI{0.72}{\meter}(\cos\SI{240}{\degree},\sin\SI{240}{\degree})$ at several frequencies. In contrast to the results in Figures \ref{Single1} and \ref{Single10}, the existence of two objects can be recognized when $f=\SI{2}{\giga\hertz}$ and only a peak of large magnitude appeared when $\mr'\in D_1$ when $f=3,\SI{4}{\giga\hertz}$. Same as previously, it is still impossible to recognize objects when $f=\SI{1}{\giga\hertz}$ and $f\geq\SI{6}{\giga\hertz}$.

\begin{figure}[h]
\begin{center}
\subfigure[$f=\SI{1}{\giga\hertz}$]{\includegraphics[width=.33\textwidth]{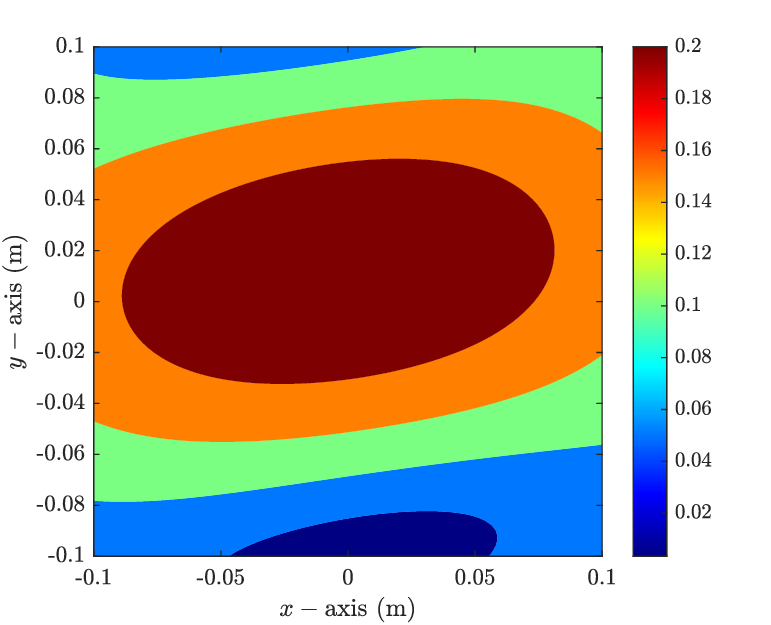}}\hfill
\subfigure[$f=\SI{2}{\giga\hertz}$]{\includegraphics[width=.33\textwidth]{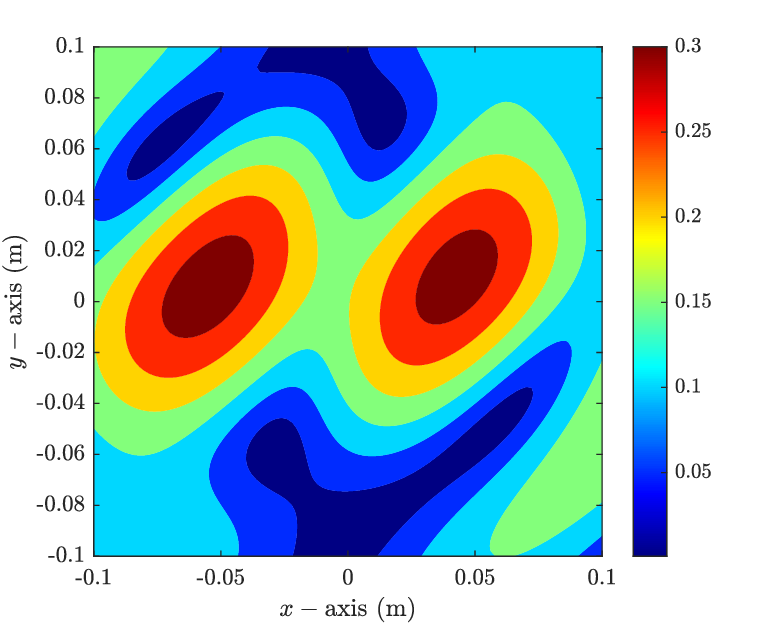}}\hfill
\subfigure[$f=\SI{3}{\giga\hertz}$]{\includegraphics[width=.33\textwidth]{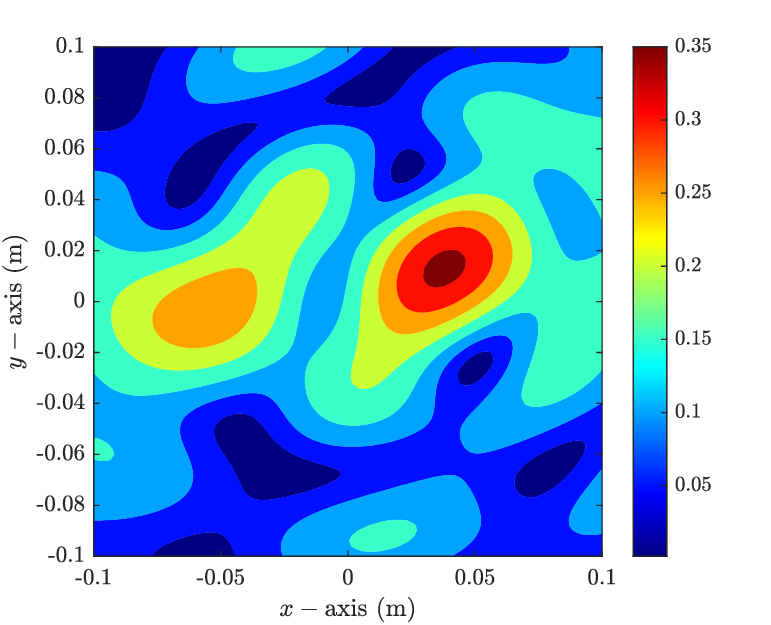}}\\
\subfigure[$f=\SI{4}{\giga\hertz}$]{\includegraphics[width=.33\textwidth]{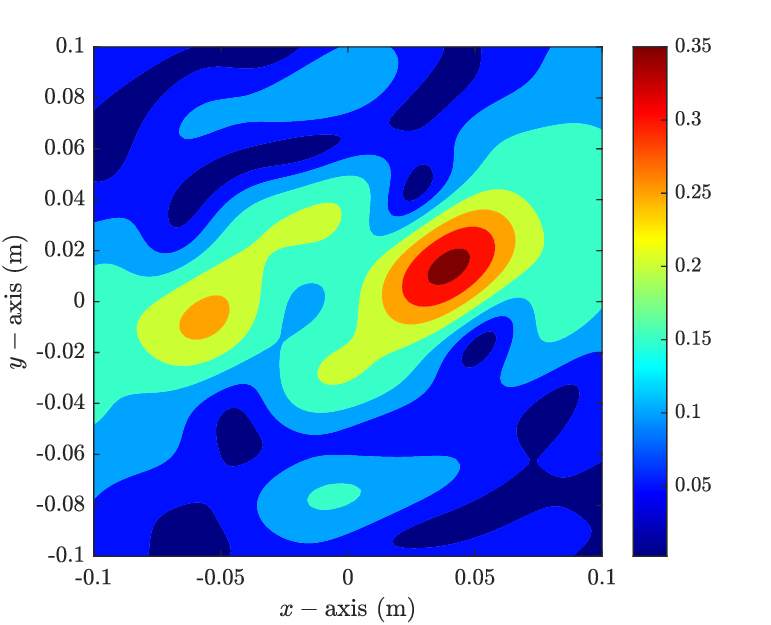}}\hfill
\subfigure[$f=\SI{6}{\giga\hertz}$]{\includegraphics[width=.33\textwidth]{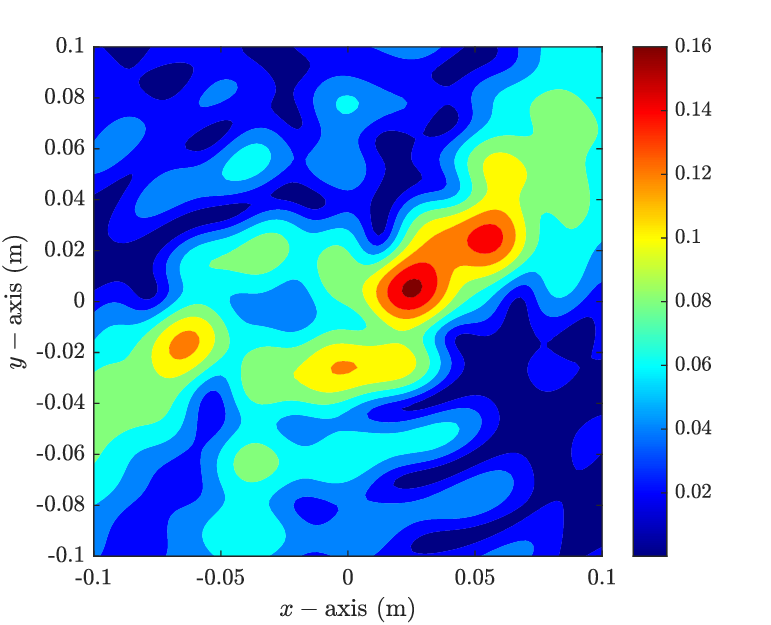}}\hfill
\subfigure[$f=\SI{8}{\giga\hertz}$]{\includegraphics[width=.33\textwidth]{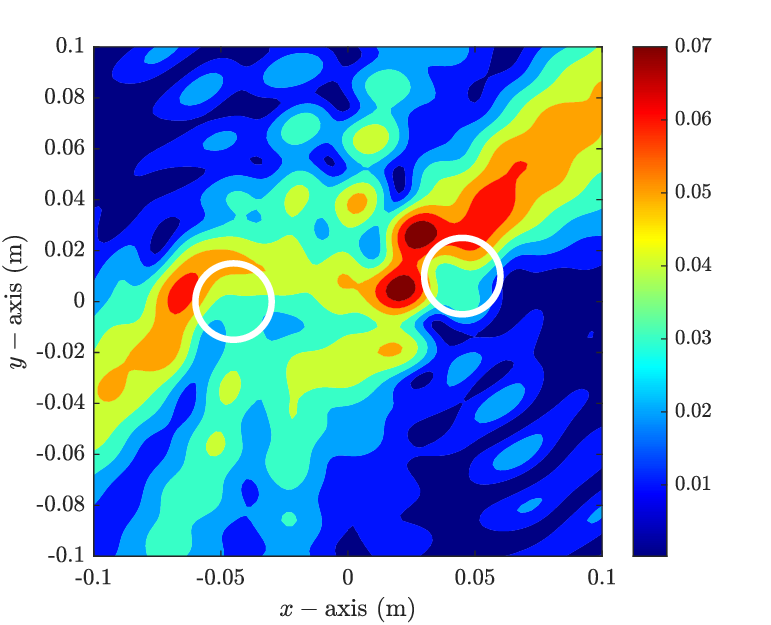}}
\caption{\label{Single25}Maps of $\mathfrak{F}_{\osm}(\mr',\ma_{25})$. White colored circles describes the boundary of objects.}
\end{center}
\end{figure}

Based on the simulation results, and Remarks \ref{remark1} and \ref{remark2}, we can conclude that the imaging performance of the OSM with single source is significantly dependent on the operated frequency and location of the emitter. Hence, design of another indicator function of the OSM seems required for a proper improvement of the imaging performance.

\section{Indicator function with multiple sources: analysis and simulation results}\label{sec:5}
Following to the several studies \cite{ACP,AGJKLSW,AGKPS,P-SUB3,P-SUB16,P-SUB18,P1}, it has been confirmed that application of multiple sources and/or frequencies successfully improves the imaging performance. Following to \cite{BIPAC,P1}, one can examine  several simulation results for the improvement of the multi-frequency OSM. Hence, we consider the application of multiple sources at a fixed frequency. Following to \cite{P1}, the following indicator function (say, OSMM) can be used: for each $\mr'\in\Omega$
\[\mathfrak{F}_{\osmm}(\mr')=\sum_{m=1}^{M}\mathfrak{F}_{\osm}(\mr',\ma_m).\]
Although one can obtain good result via the map of $\mathfrak{F}_{\osmm}(\mr')$, we introduce another indicator function to obtain a better result. To this end, we denote $\mF(\mr')$ as the following arrangement
\begin{align*}
\mF(\mr')&=\Big(\Phi(\mr',\ma_1),\Phi(\mr',\ma_2),\ldots,\Phi(\mr',\ma_N)\Big)\\
&=\frac{\kb}{6|\mb|}\begin{pmatrix}
\medskip\displaystyle\int_D\left(\frac{\eps(\mr)-\epsb}{\epsb\mub}\right)G(\mr,\ma_1)\bigg[J_0(k|\mr'-\mr|)+\frac{3}{\pi}\mathcal{E}(\mr',\mr,\ma_1)\bigg]\rd\mr\\
\displaystyle\int_D\left(\frac{\eps(\mr)-\epsb}{\epsb\mub}\right)G(\mr,\ma_2)\bigg[J_0(k|\mr'-\mr|)+\frac{3}{\pi}\mathcal{E}(\mr',\mr,\ma_2)\bigg]\rd\mr\\
\medskip\vdots\\
\displaystyle\int_D\left(\frac{\eps(\mr)-\epsb}{\epsb\mub}\right)G(\mr,\ma_m)\bigg[J_0(k|\mr'-\mr|)+\frac{3}{\pi}\mathcal{E}(\mr',\mr,\ma_m)\bigg]\rd\mr
\end{pmatrix}^T.
\end{align*}
where $\Phi(\mr',\ma_m)$ satisfies $\mathfrak{F}_{\osm}(\mr',\ma_m)=|\Phi(\mr',\ma_m)|$.
Then, based on the structure of the $\mF(\mr')$, it seems natural to test the orthogonality relation between the $\Phi(\mr',\ma_m)$ and $G(\cdot,\ma_m)$. Thus, by introducing a test vector,
\[\mH(\mr')=\Big(G(\mr',\ma_1),G(\mr',\ma_2),\ldots,G(\mr',\ma_M)\Big),\quad\mr'\in\Omega,\]
the following indicator function (say, MOSM) with multiple sources can be introduced
\[\mathfrak{F}_{\mosm}(\mr')=|\mF(\mr')\cdot\overline{\mH(\mr')}|=\left|\sum_{m=1}^{M}\Phi(\mr',\ma_m)\overline{G(\mr',\ma_m)}\right|.\]
The map of $\mathfrak{F}_{\mosm}(\mr')$ will contain peaks of large magnitude at $\mr'\in D_s$ thereby, it will be possible to recognize the existence or outline shape of $D_s$, $s=1,2,\ldots,S$. In order to discover the feasibility and some properties of the $\mathfrak{F}_{\mosm}(\mr')$, we derive the following result.

\begin{theorem}\label{OSM_Multiple}
  Let $\vv=(\cos\vartheta,\sin\vartheta)$, $\vv_m=(\cos\vartheta_m,\sin\vartheta_m)$, and $\mr'-\mr=|\mr'-\mr|(\cos\phi,\sin\phi)$. Then, for sufficiently large $M$ and $\omega$, $\mathfrak{F}_{\mosm}(\mr')$ can be represented as follows:
  \begin{equation}\label{Structure_Multiple}
    \mathfrak{F}_{\mosm}(\mr')=\left|\frac{2}{3|\ma||\mb|}\int_D\left(\frac{\eps(\mr)-\epsb}{\epsb\mub}\right)\left(J_0(\kb|\mr'-\mr|)^2+\frac{3}{\pi}\mathcal{M}(\mr',\mr)\right)\rd\mr\right|,
  \end{equation}
  where
  \[\mathcal{M}(\mr',\mr)=\sum_{p=1}^{\infty}\frac{i^p}{p}J_p(\kb|\mr'-\mr|)^2\sin\left(\frac{2p}{3}\pi\right).\]
\end{theorem}
\begin{proof}
Based on \eqref{Structure_Single}, we have
\begin{align*}
&\sum_{m=1}^{M}\Phi(\mr',\ma_m)\overline{G(\mr',\ma_m)}\\
&=\sum_{m=1}^{M}\left(\frac{\kb}{6|\mb|}\int_D\left(\frac{\eps(\mr)-\epsb}{\epsb\mub}\right)G(\mr,\ma_m)\left[J_0(\kb|\mr'-\mr|)+\frac{12}{5\pi}\mathcal{E}(\mr',\mr,\ma_m)\right]\rd\mr\right)\overline{G(\mr',\ma_m)}\\
&=\frac{\kb}{6|\mb|}\int_D\left(\frac{\eps(\mr)-\epsb}{\epsb\mub}\right)\sum_{m=1}^{M}\left(G(\mr,\ma_m)\overline{G(\mr',\ma_m)}\left[J_0(\kb|\mr'-\mr|)+\frac{12}{5\pi}\mathcal{E}(\mr',\mr,\ma_m)\right]\right)\rd\mr.
\end{align*}
Since $M$ is sufficiently large, applying \eqref{Asymptotic_Hankel} and \eqref{Jacobi-Anger} yields
\begin{align}
\begin{aligned}\label{Term1}
\sum_{m=1}^{M}G(\mr,\ma_m)\overline{G(\mr',\ma_m)}&=\sum_{m=1}^{M}\frac{2}{\kb|\ma_m|\pi}e^{i\kb\vv_m\cdot(\mr'-\mr)}\\
&=\frac{2}{\kb|\ma|\pi}\int_{\mathbb{S}^1}e^{i\kb\vv\cdot(\mr'-\mr)}\rd\vv\\
&=\frac{2}{\kb|\ma|\pi}\int_0^{2\pi}e^{i\kb|\mr'-\mr|\cos(\vartheta-\phi)}\rd\vartheta=\frac{4}{\kb|\ma|}J_0(\kb|\mr'-\mr|)
\end{aligned}
\end{align}
and
\begin{align}
\begin{aligned}\label{Term2}
&\sum_{m=1}^{M}G(\mr,\ma_m)\overline{G(\ma_m,\mr)}\cos\big(p(\vartheta_m-\phi)\big)\\
&=\frac{2}{\kb|\ma|\pi}\int_0^{2\pi}\cos\big(p(\vartheta-\phi)\big)e^{i\kb|\mr'-\mr|\cos(\vartheta-\phi)}\rd\vartheta\\
&=\frac{2}{\kb|\ma|\pi}\int_0^{2\pi}\cos\big(p(\vartheta-\phi)\big)\bigg(J_0(\kb|\mr'-\mr|)+2\sum_{q=1}^{\infty}i^qJ_q(\kb|\mr'-\mr|)\cos\big(q(\vartheta-\phi)\big)\bigg)\rd\vartheta\\
&=\frac{4i^p}{\kb|\ma|} J_p(\kb|\mr'-\mr|).
\end{aligned}
\end{align}

Based on \eqref{Term1} and \eqref{Term2}, we can examine that
\begin{multline*}
\sum_{m=1}^{M}G(\mr,\ma_m)\overline{G(\ma_m,\mr)}\left[J_0(\kb|\mr'-\mr|)+\frac{3}{\pi}\mathcal{E}(\mr',\mr,\ma_m)\right]\\=\frac{4}{\kb|\ma|}\left[J_0(\kb|\mr'-\mr|)^2+\frac{3}{\pi}\sum_{p=1}^{\infty}\frac{i^p}{p}J_p(\kb|\mr'-\mr|)^2\sin\left(\frac{2p}{3}\pi\right)\right]
\end{multline*}
and correspondingly, we can derive \eqref{Structure_Multiple}.
\end{proof}

Based on the Theorem \ref{OSM_Multiple}, we can examine some properties of the indicator function $\mathfrak{F}_{\mosm}(\mr')$.

\begin{remark}[Availability of object detection]\label{RemarkM1}
Similar to the discussion in Remark \ref{remark1}, the resulting plot of indicator function $\mathfrak{F}_{\osm}(\mr',\ma_m)$ is expected to exhibit peaks of magnitudes $\frac{2}{3|\ma||\mb|}\left(\frac{\eps_s-\epsb}{\epsb\mub}\right)\area(D_s)$ at the $\mr'=\mr\in D_s$ sought. Notice that opposite to the single source case, the imaging performance is independent to the location of emitters and receivers so that for any frequency of operation, it will be possible to recognize the existence or outline shape of objects through the map of $\mathfrak{F}_{\mosm}(\mr')$.
\end{remark}

\begin{remark}[Effect of the factor $\mathcal{M}(\mr',\mr)$]\label{RemarkM2}
Similar to the OSM with single source, the factor $\mathcal{M}(\mr',\mr)$ of \eqref{Structure_Multiple} not only does not contribute to the identification of objects but also disturb the identification by generating several artifacts. In order to examine the influence of the $\mathcal{M}(\mr',\mr)$, we consider the following quantity
\[\mathcal{D}_2(x)=\frac{3}{\pi}\left|\sum_{p=1}^{10^6}\frac{i^p}{p}J_p(\kb x)^2\sin\left(\frac{2p}{3}\pi\right)\right|.\]
This is similar to the value of $3\mathcal{M}(\mr',\mr)/\pi$ for $\mr=(0,0)$. Opposite to the discussion in Remark \ref{Remark3}, we can say that although $J_0(\kb|x|)^2$ and $\mathcal{M}(\mr',\mr)$ generate some artifacts, their magnitudes can be negligible so that their influence of object detection should be disregarded. We refer to Figure \ref{FigureBessel2} for an illustration of the oscillating properties of $J_0(\kb x)^2$ and $\mathcal{D}_2(x)$.
\end{remark}

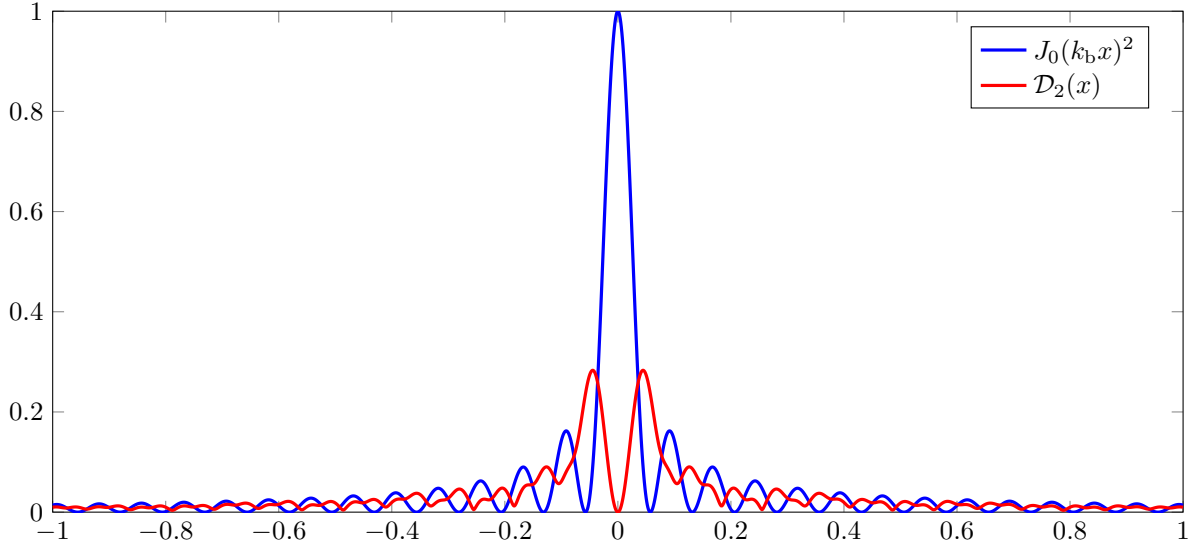
\begin{figure}[h]
\begin{center}
\begin{tikzpicture}
\begin{axis}
[legend style={fill=none},
width=\textwidth,
height=0.5\textwidth,
xmin=-1,
xmax=1,
ymin=0,
ymax=1,
legend pos=north east,
legend cell align={left}]
\addplot[line width=1.2pt,solid,color=blue] %
	table[x=x,y=y,col sep=comma]{BesselFunctions2.csv};
\addlegendentry{$J_0(\kb x)^2$};
\addplot[line width=1.2pt,solid,color=red] %
	table[x=x,y=z,col sep=comma]{BesselFunctions2.csv};
\addlegendentry{$\mathcal{D}_2(x)$};
\end{axis}
\end{tikzpicture}
\caption{\label{FigureBessel2}Plots of $J_0(\kb x)^2$ and $\mathcal{D}_2(x)$ for $-1\leq x\leq1$ and $f=\SI{2}{\giga\hertz}$.}
\end{center}
\end{figure}

\begin{remark}[Compare the imaging performance with single source OSM]\label{RemarkM3}
Here, we compare the imaging performance of the OSM with single and multiple sources. Based on the 1-D Plots of $\mathfrak{F}_{\osm}(\mr',\ma_1)$ and $\mathfrak{F}_{\mosm}(\mr')$ for $\mr'=(x,0)$ with $-1\leq x\leq 1$ in Figure \ref{FigureBessel3}, we can conclude that the $\mathfrak{F}_{\mosm}(\mr')$ will yield better images owing to less oscillation than $\mathfrak{F}_{\osm}(\mr',\ma_m)$.

For a more detailed description, since $\mathcal{E}(\mr',\mr,\ma_m)$ and $\mathcal{M}(\mr',\mr)$ are uniformly convergent, for each $\epsilon>0$, there exists $N=N(\epsilon)\in\mathbb{N}$ such that
\[\left|\mathcal{E}(\mr',\mr,\ma_m)-\sum_{p=1}^{N}\frac{(-i)^p}{p}J_p(\kb|\mr'-\mr|)\cos\big(p(\vartheta_m-\phi)\big)\sin\left(\frac{2p}{3}\pi\right)\right|<\epsilon\]
and
\[\left|\mathcal{M}(\mr',\mr)-\sum_{p=1}^{N}\frac{i^p}{p}J_p(\kb|\mr'-\mr|)^2\sin\left(\frac{2p}{3}\pi\right)\right|<\epsilon.\]

Suppose that $\mr'$ is not close to $\mr$ such that $\kb|\mr'-\mr|\gg1/4$. Then since following asymptotic form holds 
\[J_p(\kb|\mr'-\mr|)\approx\sqrt{\frac{2}{\kb|\mr'-\mr|}}\cos\left(\kb|\mr'-\mr|-\frac{p\pi}{2}-\frac{\pi}{4}+O\left(\frac{1}{\kb|\mr'-\mr|}\right)\right),\]
applying Euler–Maclaurin formula yields
\begin{align*}
|\mathcal{E}(\mr',\mr,\ma_m)|&\approx\left|\sum_{p=1}^{N}\frac{(-i)^p}{p}\sqrt{\frac{2}{\kb|\mr'-\mr|}}\cos\left(\kb|\mr'-\mr|-\frac{p\pi}{2}-\frac{\pi}{4}+O\left(\frac{1}{\kb|\mr'-\mr|}\right)\right)\sin\left(\frac{2p}{3}\pi\right)\right|\\
&\leq\sqrt{\frac{3}{2\kb|\mr'-\mr|}}\sum_{p=1}^{N}\frac{1}{p}=\sqrt{\frac{3}{2\kb|\mr'-\mr|}}\left(\ln N+\gamma+\frac{1}{2N}-c_p\right)\\
&\leq\sqrt{\frac{3}{2\kb|\mr'-\mr|}}\left(\ln N+\gamma+\frac{1}{2N}\right)
\end{align*}
and
\begin{align*}
|\mathcal{M}(\mr',\mr)|&\approx\left|\sum_{p=1}^{N}\frac{i^p}{p}\frac{2}{\kb|\mr'-\mr|}\cos^2\left(\kb|\mr'-\mr|-\frac{p\pi}{2}-\frac{\pi}{4}+O\left(\frac{1}{\kb|\mr'-\mr|}\right)\right)\sin\left(\frac{2p}{3}\pi\right)\right|\\
&\leq\frac{\sqrt{3}}{\kb|\mr'-\mr|}\sum_{p=1}^{N}\frac{1}{p}\leq\frac{\sqrt{3}}{\kb|\mr'-\mr|}\left(\ln N+\gamma+\frac{1}{2N}\right),
\end{align*}
where $c_p$ satisfies $0\leq c_p \leq(8p^2)^{-1}$ and $\gamma=0.577215665\ldots$ denotes the Euler–Mascheroni constant. Since
\[\frac{1}{\kb|\mr'-\mr|}<\sqrt{\frac{1}{2\kb|\mr'-\mr|}}\quad\text{if}\quad\kb|\mr'-\mr|\gg1/4,\]
we can explain that the factor $\mathcal{M}(\mr',\mr)$ generates less artifacts than $\mathcal{E}(\mr',\mr,\ma_m)$ because $\mathcal{M}(\mr',\mr)$ has a smaller amplitude than $\mathcal{E}(\mr',\mr,\ma_m)$. We refer to 1-D plots of $\mathcal{D}_1(x)$ and $\mathcal{D}_2(x)$ in Figures \ref{FigureBessel1} and \ref{FigureBessel2}, respectively.
\end{remark}

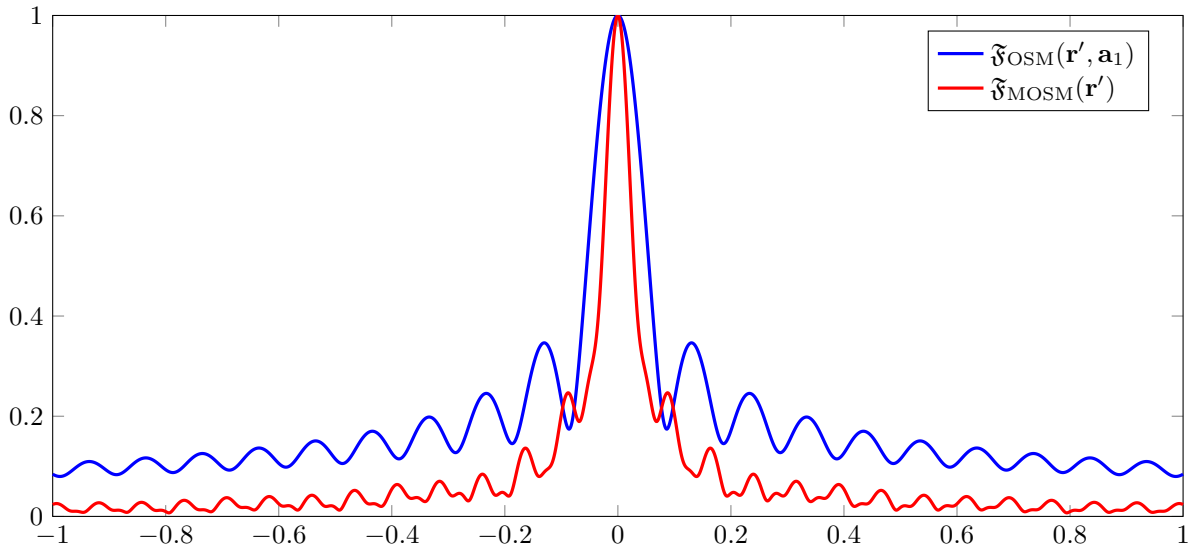
\begin{figure}[h]
\begin{center}
\begin{tikzpicture}
\begin{axis}
[legend style={fill=none},
width=\textwidth,
height=0.5\textwidth,
xmin=-1,
xmax=1,
ymin=0,
ymax=1,
legend pos=north east,
legend cell align={left}]
\addplot[line width=1.2pt,solid,color=blue] %
	table[x=x,y=y,col sep=comma]{BesselFunctions3.csv};
\addlegendentry{$\mathfrak{F}_{\osm}(\mr',\ma_1)$};
\addplot[line width=1.2pt,solid,color=red] %
	table[x=x,y=z,col sep=comma]{BesselFunctions3.csv};
\addlegendentry{$\mathfrak{F}_{\mosm}(\mr')$};
\end{axis}
\end{tikzpicture}
\caption{\label{FigureBessel3}1-D Plots of $\mathfrak{F}_{\osm}(\mr',\ma_1)$ and $\mathfrak{F}_{\mosm}(\mr')$ for $\mr=(0,0)$ and $\mr'=(x,0)$, $-1\leq x\leq 1$, at $f=\SI{2}{\giga\hertz}$.}
\end{center}
\end{figure}

Based on the Theorem \ref{OSM_Multiple} and Remark \ref{RemarkM1}, we can also obtain the following result of unique determination.

\begin{theorem}[Unique determination of small objects]\label{Theorem_Uniqueness}
Assume that the condition of the Theorem \ref{OSM_Multiple} holds. Then small objects can be identified uniquely through the map of $\mathfrak{F}_{\mosm}(\mr')$.
\end{theorem}

From now on, we consider the simulation results. The simulation configuration is same as the one in Section \ref{sec:4} except the range of emitters is from $\SI{0}{\degree}$ to $\SI{350}{\degree}$ with step size of $\SI{10}{\degree}$. Figure \ref{Multiple12} shows maps of $\mathfrak{F}_{\osmm}(\mr')$ and $\mathfrak{F}_{\mosm}(\mr')$, and Jaccard index (see \cite{Jaccard} for instance) at $f=1,\SI{2}{\giga\hertz}$. Same as the imaging result with single source, it is impossible to recognize two objects through the maps of $\mathfrak{F}_{\osmm}(\mr')$ and $\mathfrak{F}_{\mosm}(\mr')$ at $f=\SI{1}{\giga\hertz}$. Fortunately, in contrast to the single source case, the location and outline shape of two objects were successfully retrieved. Although the imaging quality of $\mathfrak{F}_{\osmm}(\mr')$ and $\mathfrak{F}_{\mosm}(\mr')$ looks similar, based on the evaluated Jaccard index, we can say that the imaging performance of $\mathfrak{F}_{\mosm}(\mr')$ is slightly better than the one of $\mathfrak{F}_{\osm}(\mr')$. Similar phenomenon can be examined through the imaging results in Figure \ref{Multiple34} at $f=3,\SI{4}{\giga\hertz}$.

It is interesting to examine that opposite to the imaging results with single source, it is possible to recognize the existence of two objects at $f=6,\SI{8}{\giga\hertz}$, refer to Figure \ref{Multiple68}. However, unfortunately, the imaging quality seems poorer than the ones in Figure \ref{Multiple34}. Hence, similar to the Remark \ref{remark1}, application of extremely high frequency is not appropriate to retrieve unknown objects.

\begin{figure}[h]
\begin{center}
\subfigure[$\mathfrak{F}_{\osmm}(\mr')$ at $f=\SI{1}{\giga\hertz}$]{\includegraphics[width=.33\textwidth]{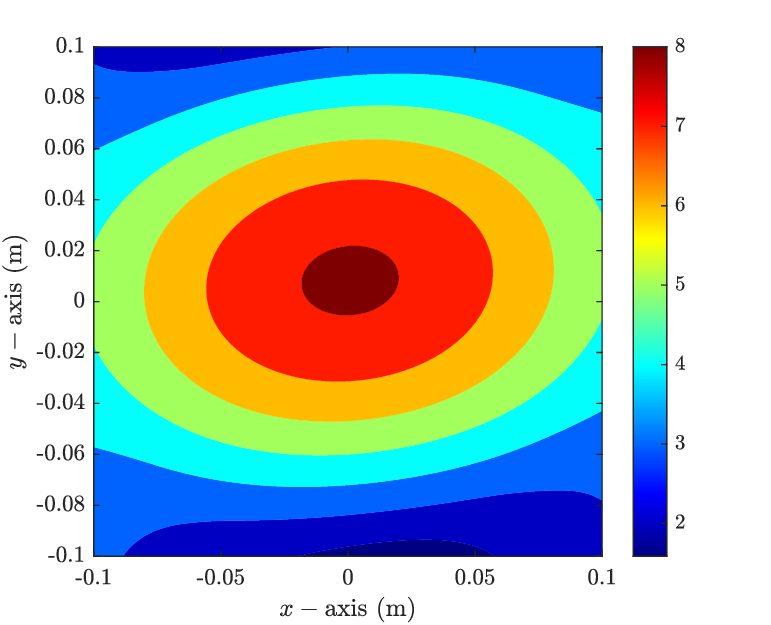}}\hfill
\subfigure[$\mathfrak{F}_{\mosm}(\mr')$ at $f=\SI{1}{\giga\hertz}$]{\includegraphics[width=.33\textwidth]{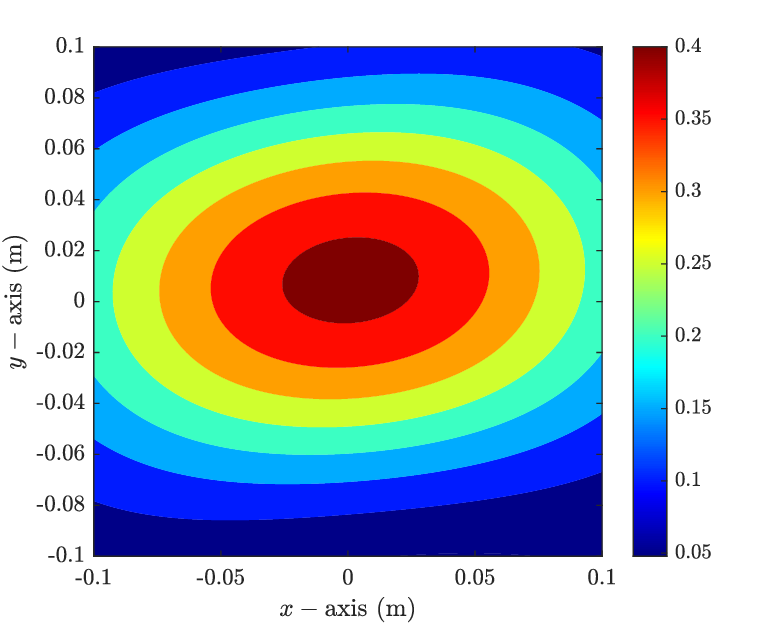}}\hfill
\subfigure[Jaccard index at $f=\SI{1}{\giga\hertz}$]{\includegraphics[width=.33\textwidth]{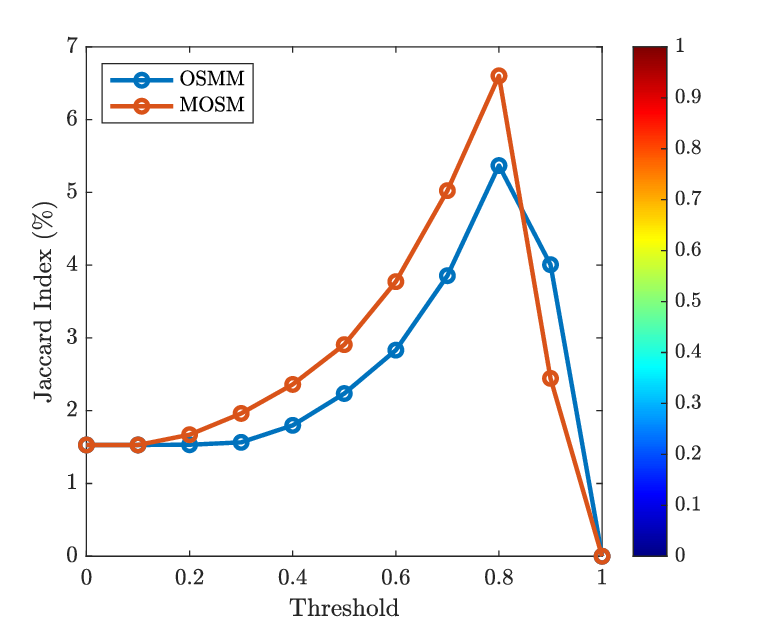}}\\
\subfigure[$\mathfrak{F}_{\osmm}(\mr')$ at $f=\SI{2}{\giga\hertz}$]{\includegraphics[width=.33\textwidth]{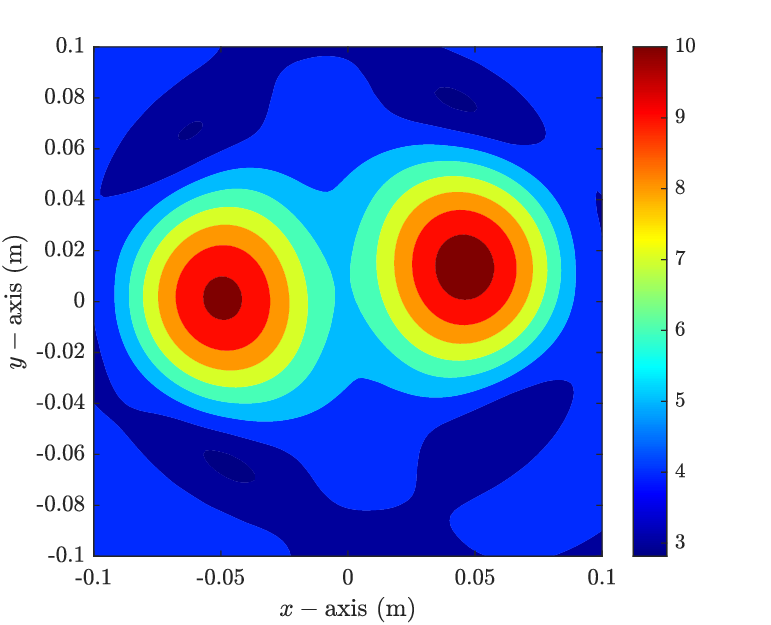}}\hfill
\subfigure[$\mathfrak{F}_{\mosm}(\mr')$ at $f=\SI{2}{\giga\hertz}$]{\includegraphics[width=.33\textwidth]{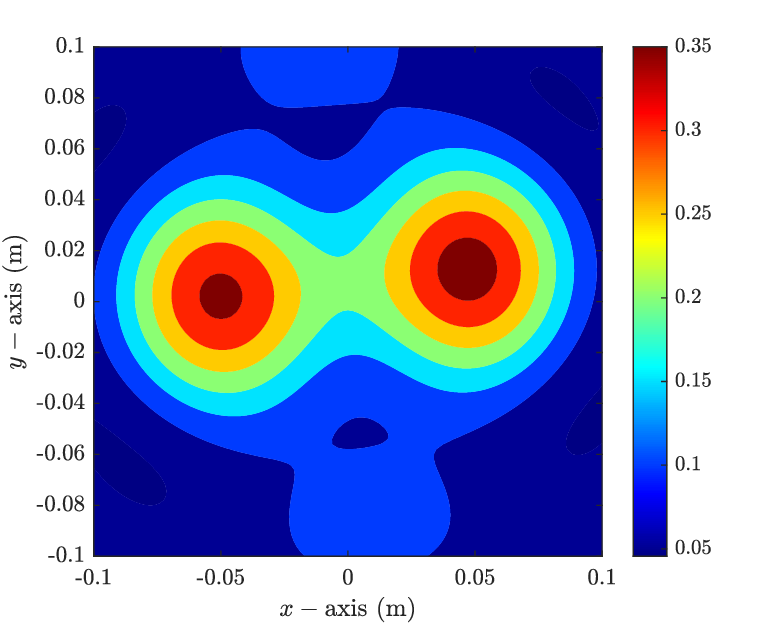}}\hfill
\subfigure[Jaccard index at $f=\SI{2}{\giga\hertz}$]{\includegraphics[width=.33\textwidth]{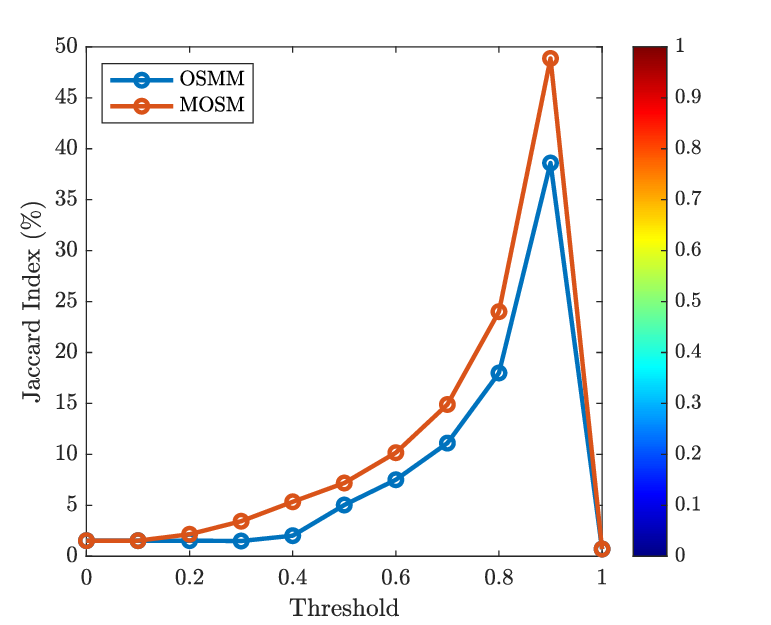}}
\caption{\label{Multiple12}Maps of $\mathfrak{F}_{\osmm}(\mr')$ and $\mathfrak{F}_{\mosm}(\mr')$, and Jaccard index versus threshold.}
\end{center}
\end{figure}

\begin{figure}[h]
\begin{center}
\subfigure[$\mathfrak{F}_{\osmm}(\mr')$ at $f=\SI{3}{\giga\hertz}$]{\includegraphics[width=.33\textwidth]{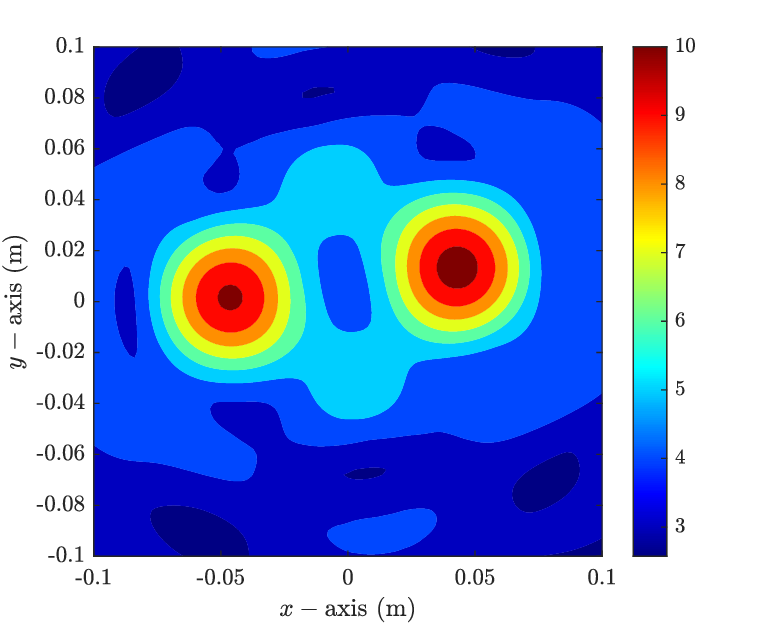}}\hfill
\subfigure[$\mathfrak{F}_{\mosm}(\mr')$ at $f=\SI{3}{\giga\hertz}$]{\includegraphics[width=.33\textwidth]{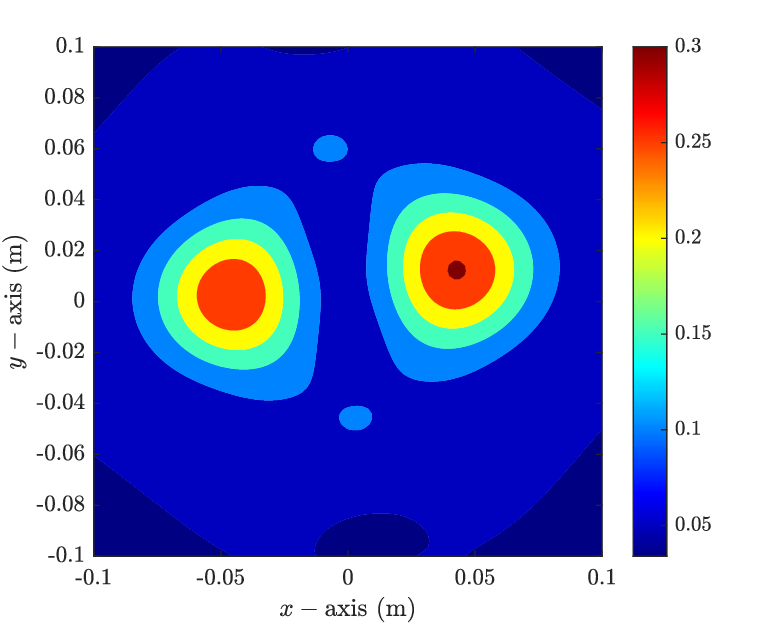}}\hfill
\subfigure[Jaccard index at $f=\SI{3}{\giga\hertz}$]{\includegraphics[width=.33\textwidth]{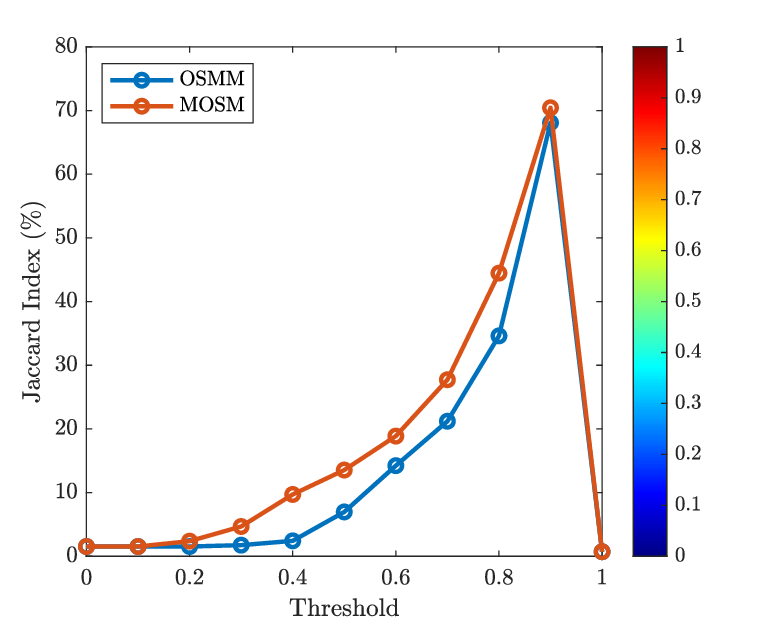}}\\
\subfigure[$\mathfrak{F}_{\osmm}(\mr')$ at $f=\SI{4}{\giga\hertz}$]{\includegraphics[width=.33\textwidth]{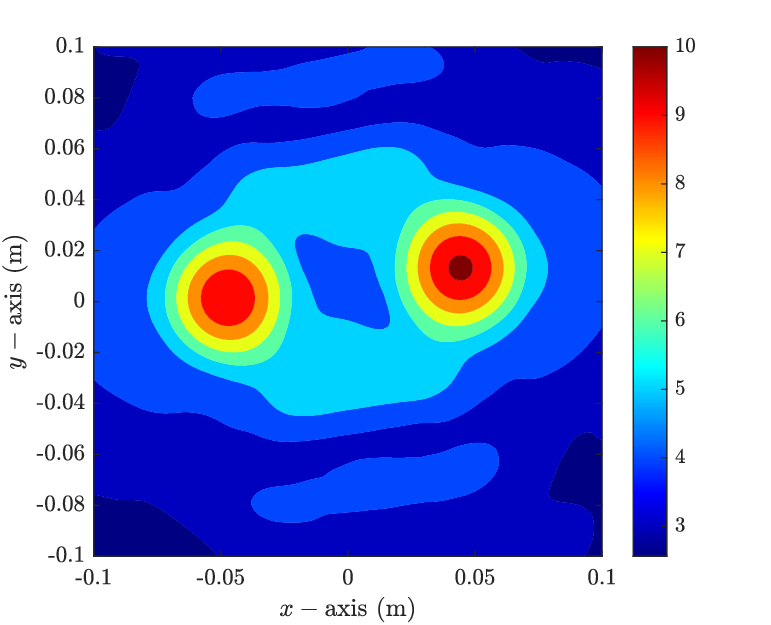}}\hfill
\subfigure[$\mathfrak{F}_{\mosm}(\mr')$ at $f=\SI{4}{\giga\hertz}$]{\includegraphics[width=.33\textwidth]{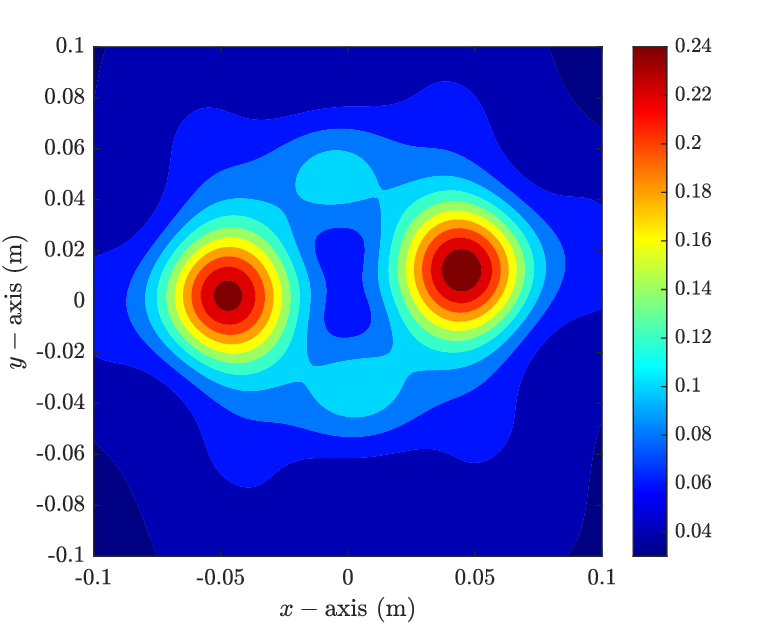}}\hfill
\subfigure[Jaccard index at $f=\SI{4}{\giga\hertz}$]{\includegraphics[width=.33\textwidth]{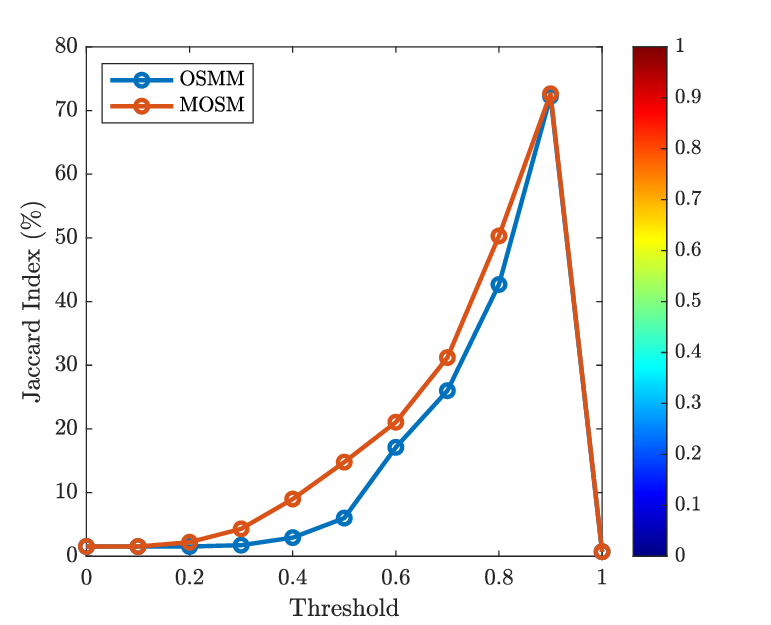}}
\caption{\label{Multiple34}Maps of $\mathfrak{F}_{\osmm}(\mr')$ and $\mathfrak{F}_{\mosm}(\mr')$, and Jaccard index versus threshold.}
\end{center}
\end{figure}

\begin{figure}[h]
\begin{center}
\subfigure[$\mathfrak{F}_{\osmm}(\mr')$ at $f=\SI{6}{\giga\hertz}$]{\includegraphics[width=.33\textwidth]{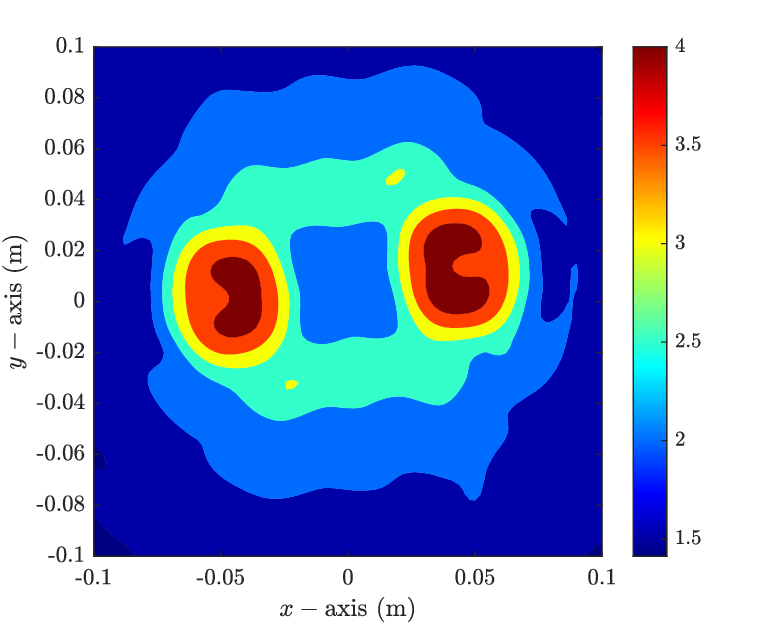}}\hfill
\subfigure[$\mathfrak{F}_{\mosm}(\mr')$ at $f=\SI{6}{\giga\hertz}$]{\includegraphics[width=.33\textwidth]{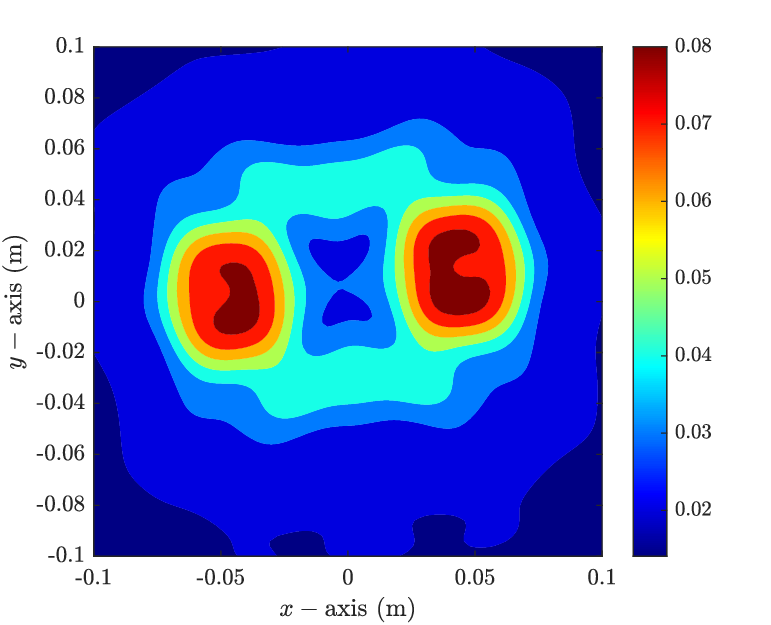}}\hfill
\subfigure[Jaccard index at $f=\SI{6}{\giga\hertz}$]{\includegraphics[width=.33\textwidth]{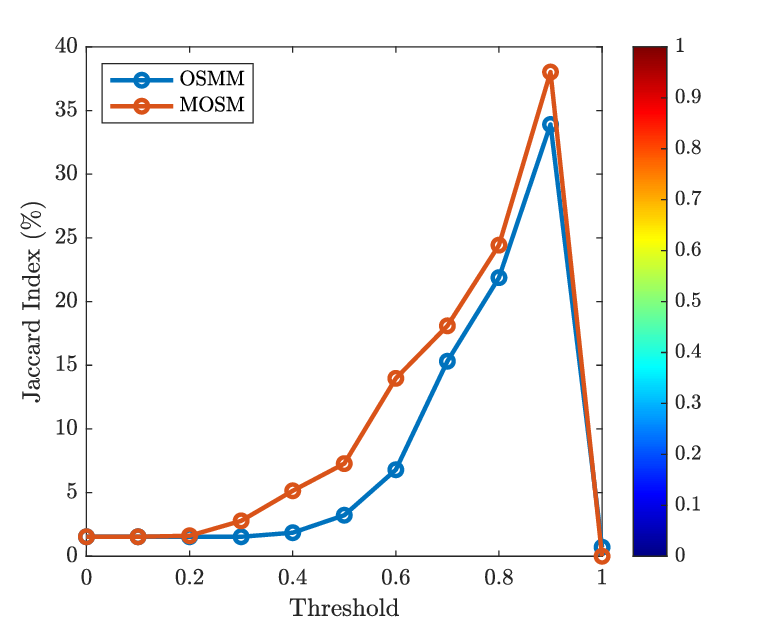}}\\
\subfigure[$\mathfrak{F}_{\osmm}(\mr')$ at $f=\SI{8}{\giga\hertz}$]{\includegraphics[width=.33\textwidth]{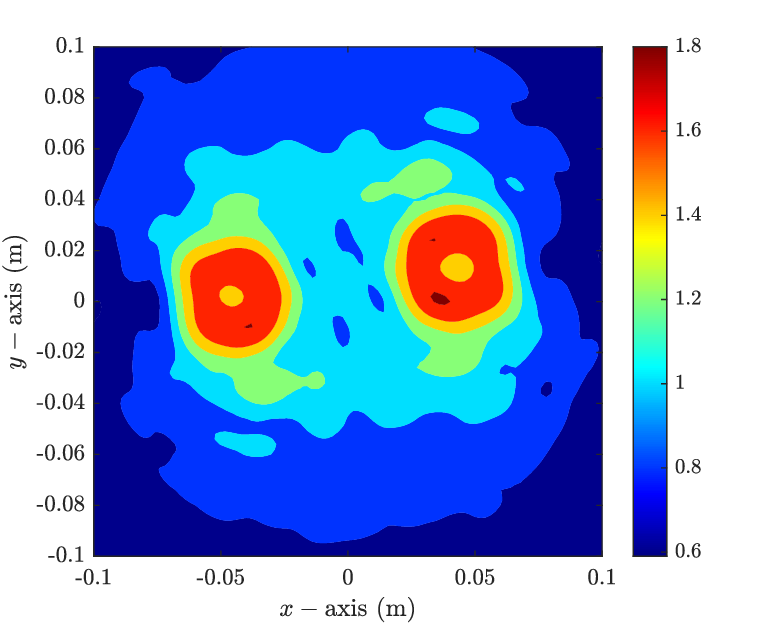}}\hfill
\subfigure[$\mathfrak{F}_{\mosm}(\mr')$ at $f=\SI{8}{\giga\hertz}$]{\includegraphics[width=.33\textwidth]{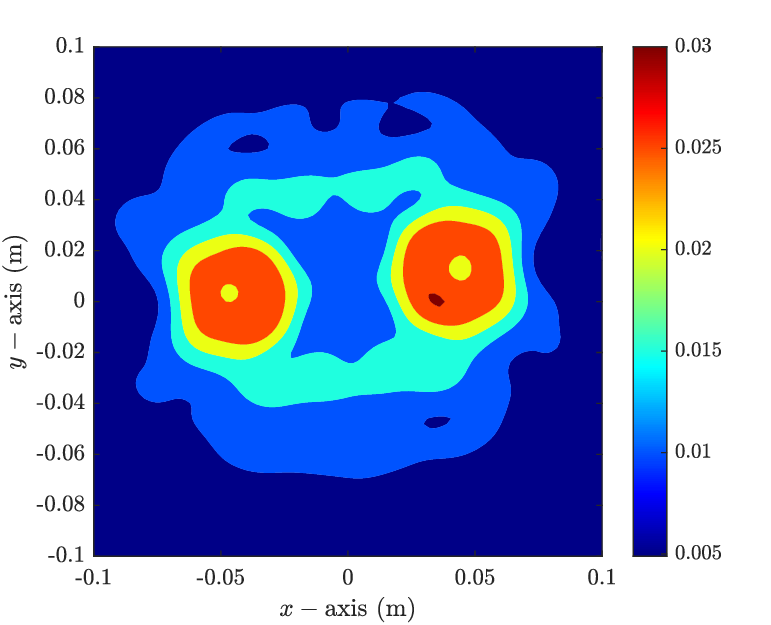}}\hfill
\subfigure[Jaccard index at $f=\SI{8}{\giga\hertz}$]{\includegraphics[width=.33\textwidth]{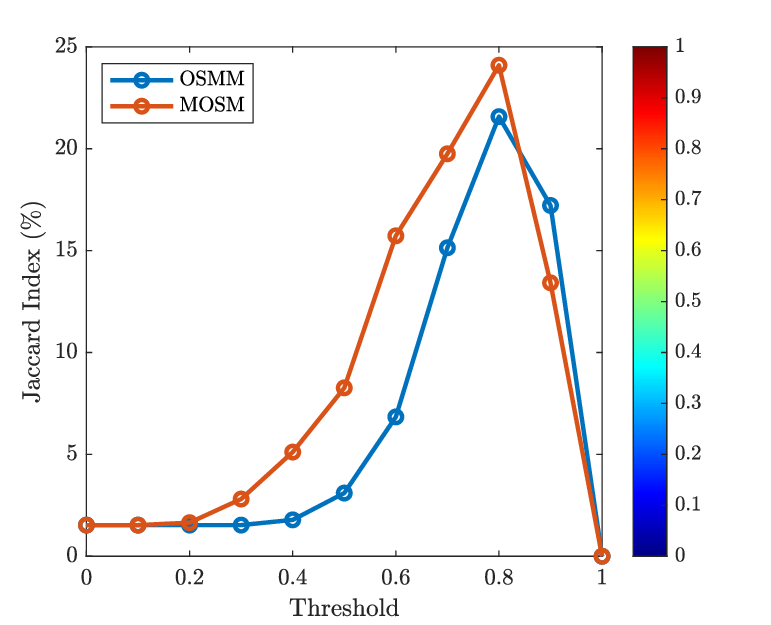}}
\caption{\label{Multiple68}Maps of $\mathfrak{F}_{\osmm}(\mr')$ and $\mathfrak{F}_{\mosm}(\mr')$, and Jaccard index versus threshold.}
\end{center}
\end{figure}

Finally, we mention that Theorem \ref{Theorem_Uniqueness} holds based on the imaging results in Figures  \ref{Multiple12}, \ref{Multiple34}, and \ref{Multiple68}.

\section{Conclusion}\label{sec:6}
In this paper, we have considered the application of the OSM with single source to identify the existence and outline shape of small dielectric objects from real-world experimental data. Thanks to the asymptotic expansion formula for the scattered field data in the presence of small object, we derived an accurate relationship between the indicator function and infinite series of the Bessel function of the first kind. On the basis of the derived relationship, we explored various properties of the indicator function including the applicability and limitation. We have also verified the theoretical result with experimental data at various frequencies.

In order to improve the imaging performance of the OSM, we have collected scattered field data with multiple sources and introduced a new indicator function. Throughout a careful analysis, we have shown that the imaging performance of the designed indicator function is independent to the location of the emitter and better than the traditional ones with single and multiple sources. Moreover, it can be possible to identify small object uniquely if applied frequency is not extremely low or high.

In this paper, we considered the application of the OSM with 2D Fresnel experimental database. An extension to the 3D Fresnel experimental database \cite{GSE} will be the forthcoming work.

\section*{Acknowledgments}
This research was supported by the National Research Foundation of Korea (NRF) grant funded by the Korea government (MSIT) (NRF-2020R1A2C1A01005221) and the research program of the Kookmin University.

\bibliographystyle{elsarticle-num-names}
\bibliography{../../../References}
\end{document}